\documentclass{amsart}

\usepackage[mathscr]{eucal}
\usepackage{amssymb}
\usepackage[usenames,dvipsnames]{color}
\usepackage[normalem]{ulem}
\usepackage{amsthm}
\usepackage{bbold}
\usepackage{enumerate}
\usepackage{array}
\usepackage{amsmath}
\usepackage{mathtools}

\usepackage{hyperref}

\hypersetup{colorlinks=true,citecolor=BrickRed,urlcolor=BrickRed,linkcolor=BrickRed}

\numberwithin{equation}{section}
\usepackage[all]{xy}

\newdir{ >}{{}*!/-10pt/\dir{>}}

\swapnumbers 

\newtheorem{Thm}[equation]{Theorem}
\newtheorem*{Thm*}{Theorem}
\newtheorem{Prop}[equation]{Proposition}
\newtheorem{Lem}[equation]{Lemma}
\newtheorem{Cor}[equation]{Corollary}
\theoremstyle{remark}
\newtheorem{Def}[equation]{Definition}
\newtheorem{Ter}[equation]{Terminology}
\newtheorem{Exa}[equation]{Example}

\newtheorem{Hyp}[equation]{Hypothesis}

\newtheorem{Rem}[equation]{Remark}

\newcommand{\nc}{\newcommand}
\nc{\dmo}{\DeclareMathOperator}

\dmo{\Ab}{Ab}
\dmo{\Abelem}{Abelem}
\dmo{\Add}{Add}
\dmo{\Aut}{Aut}
\dmo{\Bi}{bi}
\dmo{\Bisets}{Bisets}
\dmo{\CAT}{CAT}
\dmo{\coev}{coev}
\dmo{\Coloc}{Coloc}
\dmo{\ev}{ev}
\dmo{\Fib}{Fib}
\dmo{\Free}{Free}
\dmo{\Id}{Id}
\dmo{\Loc}{Loc}
\dmo{\rmI}{I}
\dmo{\rmL}{L}
\dmo{\rmR}{R}
\dmo{\Spc}{Spc}
\dmo{\Thick}{Thick}
\dmo{\chara}{char}%
\dmo{\coh}{coh} 
\dmo{\Coind}{CoInd}
\dmo{\coker}{coker}

\dmo{\cone}{cone}
\dmo{\Der}{D}
\nc{\Rder}{\mathrm{R}} 
\nc{\Lder}{\mathrm{L}} 
\dmo{\Khocat}{K}
\dmo{\End}{End}
\dmo{\Ext}{Ext}
\dmo{\rmH}{H}
\dmo{\Ho}{Ho}
\dmo{\Hom}{Hom}
\dmo{\id}{id}
\dmo{\Img}{Im}
\dmo{\incl}{incl}
\dmo{\Ind}{Ind}
\dmo{\CoInd}{CoInd}
\dmo{\Ker}{Ker}
\dmo{\Les}{Les}
\dmo{\Map}{Map}%
\dmo{\Mod}{Mod}
\dmo{\GrMod}{GrMod}
\dmo{\lax}{lax}
\dmo{\modname}{mod}%
\dmo{\grmod}{grmod}
\dmo{\Mor}{Mor}%
\dmo{\Obj}{Obj}
\dmo{\opname}{op}
\dmo{\Or}{Or}
\dmo{\pr}{pr}
\dmo{\canin}{in} 
\dmo{\Proj}{Proj} 
\dmo{\proj}{proj}
\dmo{\Qcoh}{Qcoh}
\dmo{\rank}{rank}
\dmo{\Res}{Res}
\dmo{\Rname}{R}
\dmo{\SH}{SH}
\nc{\SHp}{\SH_{(p)}}
\dmo{\smallb}{b}
\dmo{\smallperf}{perf}
\dmo{\Span}{Span}
\dmo{\Spec}{Spec}
\dmo{\Stab}{Stab}
\dmo{\stab}{stab}
\dmo{\supp}{supp}
\dmo{\switch}{switch}
\dmo{\TTR}{TTR}

\nc{\Beren}[1]{{\color{MidnightBlue}#1}}
\nc{\Ivo}[1]{{\color{OliveGreen}#1}}
\nc{\Paul}[1]{{\color{Violet}#1}}
\nc{\Pout}[1]{\Paul{\sout{#1}}}
\nc{\Bout}[1]{\Beren{\sout{#1}}}
\nc{\Iout}[1]{\Ivo{\sout{#1}}}
\nc{\IFF}{$\Leftrightarrow$}
\nc{\DO}{\omega_f}
\nc{\Crelcpt}{\cat{C}^{c/f}}
\nc{\Crich}{\underline{\cat{C}}}
\nc{\DbG}{\Db(\kk G\mmod)}
\nc{\uA}{\underline{A}}
\nc{\doublequot}[3]{#1\backslash #2/#3}
\nc{\HGK}{\doublequot HGK}
\nc{\quadtext}[1]{\quad\textrm{#1}\quad}
\nc{\qquadtext}[1]{\qquad\textrm{#1}\qquad}
\nc{\PZG}{\cat C_{\bbZ}(\bbZ G)}
\nc{\TTRK}{\TTR(\cat K)}
\nc{\psets}{\mathsf{-sets}_\sbull}
\nc{\Gsets}{G\mathsf{-sets}}
\nc{\Hsets}{H\mathsf{-sets}}
\nc{\AddK}{\Add^{\Sigma}(\cat K)}
\nc{\adj}{\dashv}
\nc{\adjto}{\rightleftarrows}
\nc{\AK}{A\MModcat{K}}
\nc{\BK}{B\MModcat{K}}
\nc{\bbA}{\mathbb{A}}
\nc{\bbB}{\mathbb{B}}
\nc{\bbC}{\mathbb{C}}
\nc{\bbI}{\mathbb{I}}
\nc{\bbN}{\mathbb{N}}
\nc{\bbP}{\mathbb{P}}
\nc{\bbQ}{\mathbb{Q}}
\nc{\bbR}{\mathbb{R}}
\nc{\bbZ}{\mathbb{Z}}
\nc{\bbZp}{\mathbb{Z}_{(p)}}
\nc{\Sphere}{\mathbb{S}} 
\nc{\cat}[1]{\mathscr{#1}}
\nc{\Displ}{\displaystyle}
\nc{\ie}{{\sl i.e.}\ }
\nc{\into}{\mathop{\rightarrowtail}}
\nc{\inv}{^{-1}}
\nc{\isoto}{\buildrel \sim\over\to}
\nc{\isotoo}{\mathop{\buildrel \sim\over\too}}
\nc{\kk}{\Bbbk}
\nc{\onto}{\mathop{\twoheadrightarrow}}
\nc{\too}{\mathop{\longrightarrow}\limits}
\nc{\xytriangle}[7]{\xymatrix@C=#7em{#1\ar[r]^-{\Displ #4} & #2 \ar[r]^-{\Displ #5}&#3\ar[r]^-{\Displ #6}&T #1}}
\nc{\ababs}{{\sl ab absurdo}}
\nc{\adh}[1]{\overline{#1}}
\nc{\adhpt}[1]{\adh{\{#1\}}}
\nc{\aka}{{a.\,k.\,a.}\ }
\nc{\ala}{{\sl \`a la}\ }
\nc{\Autcat}[1]{\Aut_{\cat #1}}
\nc{\cO}{\mathcal{O}}
\nc{\calO}{\mathcal{O}}
\nc{\cV}{\mathcal{V}}
\nc{\Db}{\Der^{\smallb}}
\nc{\Dqc}{\Der_{\Qcoh}}
\nc{\Dperf}{\Der^{\smallperf}}
\nc{\eg}{{\sl e.\,g.}}
\nc{\eps}{\epsilon}
\nc{\FFree}{\,\text{--}\Free}%
\nc{\FFreecat}[1]{\FFree_{\cat #1}}
\nc{\FK}{\mathcal{F}(\cat K)}
\nc{\gm}{\mathfrak{m}}
\nc{\Homcat}[1]{\Hom_{\cat #1}}
\nc{\Morcat}[1]{\Mor_{\cat #1}}
\nc{\hook}{\hookrightarrow}
\nc{\Idcat}[1]{\Id_{\cat{#1}}}
\nc{\ideal}[1]{\langle #1\rangle}
\nc{\ihom}{{\mathsf{hom}}} 
\nc{\ihomcat}[1]{\ihom_{\cat #1}}
\nc{\Kcat}[1]{#1\MModcat{K}}
\nc{\KP}{\cat{K}_{\cat P}}
\nc{\loccit}{{\sl loc.\ cit.}}
\nc{\lind}{\rmL\!}
\nc{\RR}{\rmR\!}
\nc{\Lotimes}{\otimes^{\rmL}}
\nc{\Mid}{\,\big|\,}
\nc{\MMod}{\,\text{-}\Mod}%
\nc{\MModcat}[1]{\MMod_{\cat #1}}%
\nc{\mmod}{\,\text{--}\modname}%
\nc{\mmodb}{\mmod^\sbull}%
\nc{\op}{^{\opname}}
\nc{\oto}[1]{\overset{#1}\to}
\nc{\otoo}[1]{\overset{#1}{\,\too\,}}
\nc{\ourfrac}[2]{\genfrac{}{}{0pt}{}{\Displ #1}{\scriptstyle #2}}
\nc{\ouriff}{\Leftrightarrow}
\nc{\oursetminus}{\!\smallsetminus\!}
\nc{\potimes}[1]{^{\otimes #1}}
\nc{\pproj}{\,\text{-}\proj}
\nc{\ptimes}[1]{^{\times #1}}
\nc{\dd}[1]{_{{\scriptscriptstyle(#1)}}}
\nc{\uu}[1]{^{{\scriptscriptstyle(#1)}}}
\nc{\pushout}{\textrm{\rm p.o.}}
\nc{\qp}{q_{_{\scriptstyle \cat P}}\!}%
\nc{\Rcat}[1]{\Rname_{\cat #1}^\sbull}
\nc{\rdto}{}
\nc{\restr}[1]{_{|_{\scriptstyle #1}}}
\nc{\RK}{\Rcat{K}}
\nc{\sbull}{{\scriptscriptstyle\bullet}}
\nc{\SET}[2]{\big\{\,#1\Mid#2\,\big\}}
\nc{\SHA}{\SH{}^{\bbA^{1}}}
\nc{\SHfin}{\SH^{\text{\rm fin}}}
\nc{\smallmatrice}[1]{\left(\begin{smallmatrix} #1 \end{smallmatrix}\right)}
\nc{\SpcAK}{\Spc(A\MModcat{K})}
\nc{\SpcK}{\Spc(\cat K)}
\nc{\suppcat}[1]{\supp(\cat #1)}
\nc{\then}{\Rightarrow}
\nc{\tideal}[1]{\ideal{#1}}
\nc{\unit}{\mathbb{1}}
\nc{\unitcat}[1]{\unit_{\cat #1}}
\nc{\onept}{\mathrm{B}} 

\nc{\HG}{\!{}^{^H}\overline{G}}
\nc{\uY}{\widetilde{Y}}

\nc{\Dk}{\dual_{\kappa}}
\nc{\Dkk}{\dual_{\kappa'}}
\nc{\bs}{\backslash}
\nc{\biCpt}{\mathrm{biCpt}}
\nc{\biLCpt}{\mathrm{biLCpt}} 
\nc{\Grps}{\mathsf{Grps}}
\nc{\Sets}{\mathsf{Sets}}
\nc{\Top}{\mathsf{Top}}
\nc{\Comp}{\mathsf{Top}^{\mathsf{comp}}}

\nc{\lG}{{}_{{\color{Gray}\scriptscriptstyle G}}}
\nc{\lH}{{}_{{\color{Gray}\scriptscriptstyle H}}}
\nc{\rG}{_{{\color{Gray}\!\scriptscriptstyle G}}}
\nc{\rH}{_{{\color{Gray}\!\scriptscriptstyle H}}}
\nc{\rK}{_{{\color{Gray}\!\scriptscriptstyle K}}}

\nc{\dual}{\Delta}

\nc{\ra}{\rightarrow}
\nc{\xra}{\xrightarrow}

\nc{\C}{\mathbb{C}} 
\nc{\Cont}{\mathrm{C}} 
\nc{\KK}{\mathsf{KK}}
\nc{\Modules}{\mathsf{Mod}}
\nc{\Alg}{\mathsf{Alg}}
\nc{\Sep}{\mathsf{Sep}}
\begin{document}


\title[Grothendieck-Neeman duality]{Grothendieck-Neeman duality\\ and the Wirthm\"uller isomorphism}
\author{Paul Balmer}
\author{Ivo Dell'Ambrogio}
\author{Beren Sanders}
\date{September 28, 2015}

\address{Paul Balmer, Mathematics Department, UCLA, Los Angeles, CA 90095-1555, USA}
\email{balmer@math.ucla.edu}
\urladdr{http://www.math.ucla.edu/$\sim$balmer}

\address{Ivo Dell'Ambrogio, Laboratoire de Math\'ematiques Paul Painlev\'e, Universit\'e de \break Lille~1, Cit\'e Scientifique -- B\^at.~M2, 59665 Villeneuve-d'Ascq Cedex, France}
\email{ivo.dellambrogio@math.univ-lille1.fr}
\urladdr{http://math.univ-lille1.fr/$\sim$dellambr}

\address{Beren Sanders, Department of Mathematical Sciences, University of Copenhagen, Universitetsparken 5, 2100 Copenhagen {\O}, Denmark}
\email{sanders@math.ku.dk}
\urladdr{http://beren.blogs.ku.dk}

\begin{abstract}
We clarify the relationship between Gro\-then\-dieck duality \`a la Neeman and the Wirthm\"uller isomorphism \`a la Fausk-Hu-May.  We exhibit an interesting pattern of symmetry in the existence of adjoint functors between compactly generated tensor-triangulated categories, which leads to a surprising trichotomy: There exist either exactly three adjoints, exactly five, or infinitely many.  We highlight the importance of so-called \emph{relative dualizing objects} and explain how they give rise to dualities on canonical subcategories.  This yields a duality theory rich enough to capture the main features of Grothendieck duality in algebraic geometry, of generalized Pontryagin-Matlis duality \`a la Dwyer-Greenless-Iyengar in the theory of ring spectra, and of Brown-Comenetz duality \`a la Neeman in stable homotopy theory.
\end{abstract}

\subjclass[2010]{18E30; 14F05, 55U35}
\keywords{Grothendieck duality, ur-Wirthm\"uller, dualizing object, adjoints, compactly generated triangulated category, Serre functor}

\thanks{First-named author partially supported by NSF grant~DMS-1303073.}
\thanks{Second-named author partially supported by the Labex CEMPI (ANR-11-LABX-0007-01)}
\thanks{Third-named author partially supported by the Danish National Research Foundation through the Centre for Symmetry and Deformation (DNRF92)}

\maketitle


\vskip-\baselineskip\vskip-\baselineskip
\tableofcontents
\vskip-\baselineskip\vskip-\baselineskip\vskip-\baselineskip

\section{Introduction and statement of results}
\medskip

\subsection*{A tale of adjoint functors}
Consider a tensor-exact functor $f^*: \cat D\to \cat C$ between tensor-triangulated categories. As the notation~$f^*$ suggests, one typically obtains such functors by pulling-back representations, sheaves, spectra, etc., along some suitable ``underlying" map $f:X\to Y$ of groups, spaces, schemes, etc. (The actual underlying map~$f$ is not relevant for our discussion. Moreover, our choice of notation~$f^*$ is geometric in spirit, \ie dual to the ring-theoretic one; see Example~\ref{Exa:AG_affine}.) We are interested in the existence of adjoints to~$f^*$ and of further adjoints to these adjoints, and so on:
\begin{equation} \label{Eq:five_adjoints}
\vcenter{\xymatrix{
\cat C
  \ar@<-60pt>@{}[d]|-{\cdots}
  \ar@<-40pt>@{<-}[d]
  \ar@<-20pt>[d]
  \ar@{<-}[d]^-{\!\Displ f^{*}}
  \ar@<22pt>[d]^-{\Displ f_*}
  \ar@<45pt>@{<-}[d]
  \ar@<67pt>[d]
  \ar@<85pt>@{}[d]|-{\cdots}
    \\
\cat D
}}
\end{equation}
Such questions arise in examples because certain geometric properties of the underlying $f:X\to Y$ can sometimes be translated into the existence, or into properties, of such adjoints. This is illustrated for instance in Neeman's approach to Grothendieck duality~\cite{Neeman96}. Our main motivation is to provide a systematic treatment of these adjoints in the context of compactly generated categories, while simultaneously clarifying the relationship between so-called Wirthm\"uller isomorphisms and Grothen\-dieck duality. In that respect, our work is a continuation of Fausk-Hu-May~\cite{FauskHuMay03}. It turns out that the more adjoints exist, the more strongly related they must be to each other. Also remarkable is the existence of a tipping point after which there must exist infinitely many adjoints on both sides. This will happen for instance as soon as we have the six consecutive adjoints pictured in~\eqref{Eq:five_adjoints} above.

\smallbreak

Let us be more precise. Here is our basic set-up:

\begin{Hyp}
\label{hyp:base}%
Throughout the paper, we assume that both tensor-triangulated categories~$\cat C$ and~$\cat D$ are \emph{rigidly-compactly generated}. See Section~\ref{se:Brown} for details. In short, this means that $\cat C$ admits arbitrary coproducts, its compact objects coincide with the rigid objects (\aka the strongly dualizable objects) and $\cat C$ is generated by a \emph{set} of those rigid-compacts; and similarly for~$\cat D$. Such categories are the standard ``big" tensor-triangulated categories in common use in algebra, geometry and homotopy theory. They are the \emph{unital algebraic stable homotopy categories} of~\cite{HoveyPalmieriStrickland97} (with ``algebraic" understood broadly since it includes, for example, the topological stable homotopy category~$\SH$). See Examples~\ref{exa:schemes}--\ref{exa:BraveNew}.

Moreover, we assume that $f^*: \cat D\to \cat C$ is a tensor-exact functor (\ie strong symmetric monoidal and triangulated) \emph{which preserves arbitrary coproducts}. These hypotheses are quite natural and cover standard situations; see Examples~\ref{exa:alg_geom_general}--\ref{Exa:Brave_New} and~\ref{Exa:eqhtpy}--\ref{Exa:mothtpy}. (Such~$f^*$ are called \emph{geometric functors} in~\cite[Def.\ 3.4.1]{HoveyPalmieriStrickland97}.)
\end{Hyp}

By Neeman's Brown Representability Theorem, these basic hypotheses already imply the existence of two layers of adjoints to the right of the given~$f^*:\cat D\to \cat C$.

\begin{Thm}[Cor.\,\ref{cor:base}]
\label{thm:base}%
Under Hypothesis~\ref{hyp:base}, the functor $f^*: \cat D\to \cat C$ admits a right adjoint $f_*:\cat C\to \cat D$, which itself admits a right adjoint $f\uu{1}:\cat D\to \cat C$.
Moreover, we have a projection formula $d \otimes f_*(c) \cong f_*(f^*(d)\otimes c)$ and a couple of other relations detailed in Proposition~\ref{Prop:right_proj_formula}.
\end{Thm}

In other words, we get $f^*\adj f_*\adj f\uu{1}$ essentially ``for free''\!. This includes the unconditional existence of a special object that we want to single out:
\begin{Def}
\label{def:DO}%
Writing $\unit$ for the $\otimes$-unit, the object~$\DO:=f\uu{1}(\unit)$ in~$\cat C$ will be called the \emph{relative dualizing object} (for $f^*:\cat D\to \cat C$) in reference to the dualizing complexes of algebraic geometry; see~\cite{Lipman09} and~\cite{Neeman96,Neeman10}. This object $\DO$ of~$\cat C$ is uniquely characterized by the existence of a natural isomorphism
\begin{equation}
\label{eq:DO}%
\Homcat{D}(f_*(-),\unit)\cong\Homcat{C}(-,\DO).
\end{equation}
Equivalently, $\DO$ is characterized by the existence of a natural isomorphism
\begin{equation}
\label{eq:i-DO}%
\ihomcat{D}(f_*(-),\unit)\cong f_*\,\ihomcat{C}(-,\DO),
\end{equation}
where $\ihomcat{C}$ and $\ihomcat{D}$ are the internal hom functors on~$\cat C$ and~$\cat D$ respectively.
In other words, $\DO$ allows us to describe the usual (untwisted) dual $\dual:=\ihom(-,\unit)$ of the direct image $f_*$ as the direct image of the $\DO$-twisted dual~$\dual_{\DO}:=\ihom(-,\DO)$.
\end{Def}

\smallbreak

Armed with this object~$\DO\in\cat C$, we return to our three functors~$f^*\adj f_*\adj f\uu{1}$. We prove that the existence of one more adjoint on either side forces adjoints on both sides~$f\dd{1}\adj f^*\adj f_*\adj f\uu{1}\adj f\dd{-1}$ and strong relations between these five functors. This is one of the main clarifications of the paper.

\begin{Thm} [Grothendieck-Neeman Duality, Theorem~\ref{thm:GN}] \label{thm:GN-intro}
Let $f^*: \cat D\to \cat C$ be as in our basic Hypothesis~\ref{hyp:base} and consider the automatic adjoints $f^*\adj f_*\adj f\uu{1}$ (Thm.\,\ref{thm:base}).
Then the following conditions are equivalent:
\begin{enumerate}[\indent\rm(GN1)]
\smallbreak
\item
\label{it:Groth}%
\emph{Grothendieck duality:}
There is a natural isomorphism
%
$$
\DO \otimes f^*(-) \cong f\uu{1}(-).
$$
%
\smallbreak
\item\label{it:Neeman}
	\emph{Neeman's criterion:} The functor $f_*$ preserves compact objects, or equivalently its right adjoint $f\uu{1}$ preserves coproducts, or equivalently by Brown Representability $f\uu{1}$ admits a right adjoint~$f\dd{-1}$.
\smallbreak
\item\label{it:new}
The original functor $f^*:\cat D\to \cat C$ preserves products, or equivalently by Brown Representability $f^*$ admits a left adjoint $f\dd{1}$.
\end{enumerate}
\smallbreak
\noindent Moreover, when these conditions hold, the five functors $f\dd{1}\adj f^*\adj f_*\adj f\uu{1}\adj f\dd{-1}$
$$
\xymatrix@R=3em{
\cat C
  \ar@{->}@<-15pt>[d]   _-{\!\Displ f\dd{1}}
  \ar@{<-}[d]|-{\!\Displ f^{*}\vphantom{f^f_f}}
  \ar@<15pt>[d]|-{\Displ f_*\vphantom{f^f_f}}
  \ar@<30pt>@{<-}[d]|-{\Displ f\uu{1}\!\!\!\vphantom{f^f_f}}
  \ar@{->}@<45pt>[d]^-{\!\Displ f\dd{-1}}
    \\
\cat D
}
$$
are related by an armada of canonical isomorphisms, detailed in Theorem~\ref{thm:GN} and Example~\ref{Rem:computer}. Most notably, we have what we call the \emph{ur-Wirthm\"uller isomorphism}
\begin{equation}
\label{eq:ur-Wirth}%
f\dd{1}(-) \cong f_*(\DO\otimes -)
\end{equation}
and we have a canonical isomorphism $\unit_{\cat C} \cong \ihomcat{C}(\DO,\DO)$.
\end{Thm}

The equivalence between~(GN\,\ref{it:Groth}) and~(GN\,\ref{it:Neeman}) was established by Neeman~\cite{Neeman96}. We name the theorem after him since he has been the main architect of compactly generated categories and since several of our techniques have been pioneered by him, in the algebro-geometric context, like in~\cite{Neeman10}. Our main input is to show that Grothendieck-Neeman duality can be detected on the original functor~$f^*$, namely by the property~(GN\,\ref{it:new}) that $f^*$ preserves products. In other words, the existence of Neeman's right adjoint~$f\dd{-1}$ on the far-right is equivalent to the existence of a left adjoint $f\dd{1}$ four steps to the left. Our Lemma~\ref{Lem:trick_with_adjoints} is the tool which allows us to move from left to right via the duality on the subcategory of compact objects. This lemma is the key to the proof of the new implication~(GN\,\ref{it:new})$\then$(GN\,\ref{it:Neeman}) above and appears again in the proof of Theorem~\ref{thm:Wirthmueller} below.

The ur-Wirthm\"uller formula~\eqref{eq:ur-Wirth} plays a fundamental role in our approach and connects with similar formulas in~\cite{FauskHuMay03}, as discussed in Remark~\ref{rem:Wirth} below.
In algebraic geometry, an isomorphism as in~\eqref{eq:ur-Wirth} is mentioned in~\cite[Rem.\,4.3]{Neeman10}.

\smallskip

Our Grothendieck-Neeman Duality Theorem~\ref{thm:GN-intro} leaves one question open, made very tempting by the isomorphism $\ihom(\DO,\DO)\cong\unit$: When is the relative dualizing object~$\DO$ $\otimes$-invertible? Amusingly, this is related to another layer of adjoints, on either side of $f\dd{1}\adj f^*\adj f_*\adj f\uu{1}\adj f\dd{-1}$. We reach here the tipping point from which infinitely many adjoints must exist on both sides.

\begin{Thm}[Wirthm\"uller Isomorphism; see Section~\ref{se:Wirth}]
\label{thm:Wirthmueller}%
Suppose that we have the five adjoints $f\dd{1}\dashv f^*\dashv f_*\dashv f\uu{1}\dashv f\dd{-1}$ of Grothendieck-Neeman duality (Thm.\,\ref{thm:GN-intro}). Then the following conditions are equivalent:
\begin{enumerate}[\indent\rm({W}1)]
\smallbreak
\item \label{it:far-left}%
	The left-most functor $f\dd{1}$ admits itself a left adjoint, or equivalently by Brown Representability it preserves arbitrary products.
\smallbreak
\item \label{it:far-right}%
	The right-most functor $f\dd{-1}$ admits itself a right adjoint, or equivalently by Brown Representability it preserves arbitrary coproducts, or equivalently its left adjoint $f\uu{1}$ preserves compact objects.
\smallbreak
\item \label{it:w-cpt}
The relative dualizing object $\DO$ (Def.\,\ref{def:DO}) is a compact object of~$\cat C$.
\smallbreak
\item \label{it:w-inv}%
The relative dualizing object $\DO$ is $\otimes$-invertible in~$\cat C$.
\smallbreak
\item \label{it:Wirth}%

There exists a \emph{(strong) Wirthm\"uller isomorphism} between $f_*$ and $f\dd{1}$; that is, there exists a $\otimes$-invertible object $\omega\in\cat C$ such that $f\dd{1}\cong f_*(\omega\otimes-)$, or equivalently such that $f_*\cong f\dd{1}(\omega\inv\otimes-)$.
\smallbreak
\item \label{it:ad-lib}
There exists an infinite tower of adjoints on both sides:
\[\vcenter{\xymatrix{
\cat C
  \ar@<-130pt>@{}[d]|-{\cdots}
  \ar@<-120pt>@{<-}[d]^-{\Displ f\uu{-n}\!}
  \ar@<-90pt>@{->}[d]^-{\Displ f\dd{n}\!}
  \ar@<-65pt>@{}[d]|-{\cdots}
  \ar@<-54pt>@{<-}[d]^-{\Displ f\uu{-1}\!}
  \ar@<-27pt>@{->}[d]^-{\Displ f\dd{1}\!}
 \ar@{<-}[d]^(.43){\Displ f^{*}}
  \ar@<27pt>@{->}[d]^-{\!\Displ f_*}
  \ar@<52pt>@{<-}[d]^-{\!\Displ f\uu{1}}
  \ar@<77pt>@{->}[d]^-{\!\Displ f\dd{-1}}
  \ar@<105pt>@{<-}[d]^-{\!\Displ f\uu{2}}
  \ar@<122pt>@{}[d]^-{\cdots}
  \ar@<140pt>@{<-}[d]^-{\!\Displ f\uu{n}}
  \ar@<165pt>@{->}[d]^-{\!\Displ f\dd{-n}}
  \ar@<195pt>@{}[d]|-{\cdots}
    \\
\cat D
}}
\]
which necessarily preserve all coproducts, products and compact objects.
\end{enumerate}
Moreover, when these conditions hold, the tower of adjoints appearing in~(W\ref{it:ad-lib}) is necessarily given for all $n\in \bbZ$ by the formulas
\begin{equation}
\label{eq:twists}%
f\uu{n}=\DO\potimes{n}\otimes f^*
\qquadtext{and}
f\dd{n}=f_*(\DO\potimes{n}\otimes-)\,.
\end{equation}
Finally, (W\ref{it:far-left})-(W\ref{it:ad-lib}) hold true as soon as the functor~$f_*:\cat C\to \cat D$ satisfies, in addition to Grothendieck-Neeman duality, any one of the following three properties:
\begin{enumerate}[\indent\rm(1)]
\item The functor $f_*$ is faithful (\ie $f^*$ is surjective up to direct summands).
\item The functor $f_*$ detects compact objects: any $x\in \cat C$ is compact if $f_*(x)$ is.
\item \label{it:rel-cpt}
Any $x\in \cat C$ is compact if $f_*(x\otimes y)$ is compact for every compact $y\in\cat C$.
\end{enumerate}
(These conditions are ordered in increasing generality, for (1)$\Rightarrow$(2)$\Rightarrow$(3).)
\end{Thm}

\begin{Rem}
\label{rem:notation}%
We opted for the notation $f\uu{n}\adj f\dd{-n} \adj f\uu{n+1}$ after trying everything else. As is well-known, notations of the form $f^!$, $f_!$, $f^\times$, $f_\#$, etc., have flourished in various settings, sometimes with contradictory meanings.  Instead of risking collision, we propose a systematic notation which allows for an infinite tower of adjoints, following the tradition that $f\uu{n}$ is numbered with $n$ going up $\cdots f\uu{n}, f\uu{n+1}\cdots$ and $f\dd{n}$ with $n$ going down $\cdots f\dd{n}, f\dd{n-1}\cdots$. Our notation also recalls that $f\uu{n}$ and $f\dd{n}$ are $n$-fold \emph{twists} of $f\uu{0}=f^*$ and $f\dd{0}=f_*$ by~$\DO$; see~\eqref{eq:twists}.
\end{Rem}

\begin{Rem}
\label{rem:Wirth}%
In the literature, Property~(W\ref{it:Wirth}) is usually simply called a Wirth\-m\"uller isomorphism, referring to the original~\cite{Wirthmueller74}. Such a strong relation between the left and right adjoints to~$f^*$ is very useful, for then $f_*$ and $f\dd{1}$ will share all properties which are stable under pre-tensoring with an invertible object (\eg, being full, faithful, etc.). Similarly, most formulas valid for one of them will easily transpose into a formula for the other one. Here, we sometimes add the adjective ``strong" to avoid collision with Fausk-Hu-May's slightly different notion of ``Wirthm\"uller context"~\cite{FauskHuMay03}; see more in Remark~\ref{rem:FHM}. Let us point out that the existence of any Wirthm\"uller isomorphism~(W\ref{it:Wirth}) is not independent of Grothendieck duality but actually requires it. This is because our new condition~(GN\,\ref{it:new}) tells us that the mere existence of the left adjoint $f\dd{1}$ forces Grothendieck duality. Furthermore, the Wirthm\"uller isomorphism itself and the twisting object~$\DO$ are borrowed from the earlier ur-Wirthm\"uller isomorphism~\eqref{eq:ur-Wirth}, since ur-Wirthm\"uller~\eqref{eq:ur-Wirth} clearly implies Wirthm\"uller~(W\ref{it:Wirth}) when $\DO$ is invertible.
\end{Rem}


In conclusion, we have the following picture:
$$
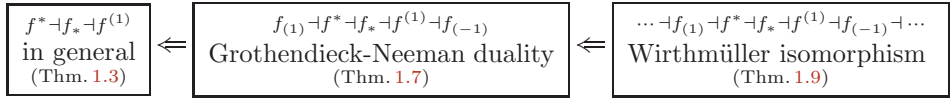
\kern-.2em
\xymatrix@C=1em{
\boxed{\ourfrac{\scriptstyle f^*\adj f_*\adj f\uu{1}}{\ourfrac{\textrm{in general}}{\textrm{(Thm.\,\ref{thm:base})}}}}
&
\boxed{\ourfrac{\scriptstyle f\dd{1}\adj f^*\adj f_*\adj f\uu{1}\adj f\dd{-1}}{\ourfrac{\textrm{Grothendieck-Neeman duality}}{\textrm{(Thm.\,\ref{thm:GN-intro})}}}}
 \ar@{=>}[l]
&
\boxed{\ourfrac{\scriptstyle \cdots \,\adj f\dd{1}\adj f^*\adj f_*\adj f\uu{1}\adj f\dd{-1}\adj\, \cdots}{\ourfrac{\textrm{Wirthm\"uller isomorphism }}{\textrm{(Thm.\,\ref{thm:Wirthmueller})}}}}
 \ar@{=>}[l]
}
$$
\begin{Cor}[Trichotomy of adjoints]
\label{cor:tricho}%
If $f^*$ is a coproduct-preserving tensor triangulated functor between rigidly-compactly generated tensor triangulated categories, then \emph{exactly one} of the following three possibilities must hold:
\begin{enumerate}[\indent\rm(1)]
\item
\label{it:gen}%
There are two adjunctions as follows and no more:
$\hphantom{f\dd{1}\!\dashv\! }f^*\dashv f_* \dashv f\uu{1}.\hphantom{ \!\dashv\! f\dd{-1}}$
\smallbreak
\item
There are four adjunctions as follows and no more:
$f\dd{1}\!\dashv\! f^*\!\dashv\! f_* \!\dashv\! f\uu{1} \!\dashv\! f\dd{-1}$.
\smallbreak
\item
\label{it:W}%
There is an infinite tower of adjunctions in both directions:
\[
\cdots f\uu{-1} \dashv
f\dd{1} \dashv
f^* \dashv
f_* \dashv
f\uu{1} \dashv
f\dd{-1}
\cdots
f\uu{n} \dashv
f\dd{-n} \dashv
f\uu{n+1}
\cdots
\]
\end{enumerate}
\end{Cor}

\begin{Rem}
\label{rem:warning}%
The dualizing object~$\DO$ could be invertible even in case~\eqref{it:gen} above, \ie \emph{without} Grothendieck-Neeman duality. See Example~\ref{Exa:warning}. Of course, there is no Wirthm\"uller isomorphism in such cases, since $f\dd{1}$ does not even exist, by~(GN\,\ref{it:new}).
\end{Rem}

\begin{Rem}
In case~\eqref{it:W}, the invertible object~$\DO$ can be trivial: $\DO\simeq\unit$. This happens precisely when~$f^*$ is a \emph{Frobenius functor}~\cite{Morita65}, \ie admits a simultaneous left-and-right adjoint
$f^* \dashv f_*\dashv f^*$. This is also called an \emph{ambidextrous adjunction}.
\end{Rem}

\begin{center}
*\ *\ *
\end{center}

\subsection*{Abstract Grothendieck duality}
In the literature, the phrase ``Groth\-en\-dieck duality" can refer to several different things. In its crudest form, it is the isomorphism $\DO\otimes f^*\cong f\uu{1}$ of~(GN\,\ref{it:Groth}) -- hence the name \emph{twisted inverse image} for~$f\uu{1}$. ``Grothendieck duality" can also refer to the compatibility $\dual\circ f_*  =f_* \circ \dual_{\DO}$ given in~\eqref{eq:i-DO}, between direct image $f_*$ and the two dualities~$\dual=\ihom(-,\unit)$ and $\dual_{\DO}=\ihom(-,\DO)$. However, this is usually formulated for certain proper subcategories $\cat C_0 \subset \cat C$ and $\cat D_0 \subset \cat D$ on which these duality functors earn their name by inducing equivalences $\cat C_0^{\op}\stackrel{\sim}{\to} \cat C_0$ and $\cat D_0^{\op}\stackrel{\sim}{\to} \cat D_0$. Then ``Grothendieck duality" refers to the situation where the functor $f_*:\cat C\to \cat D$ maps $\cat C_0$ to~$\cat D_0$ and intertwines the two dualities. The initial example was $\cat C_0=\Db(\coh X)$ and $\cat D_0=\Db(\coh Y)$ for a suitable morphism of schemes~$f:X\to Y$; here $X$ and $Y$ are assumed noetherian and $\Db(\coh X)$ is the bounded derived category of coherent $\cO_X$-modules. Since in general $\ihomcat{C}(-,\unit)$ might not preserve $\cat C_0$ (as in the geometric example just mentioned), one should also try to replace the naive duality $\dual=\ihom(-,\unit)$ of~\eqref{eq:i-DO} by a more friendly one, say $\dual_\kappa:=\ihom(-,\kappa)$ for some object~$\kappa\in\cat C_0$  having the property that $\ihomcat{C}(-,\kappa):\cat C_0\op\isoto \cat C_0$ is an equivalence.  In algebraic geometry, for $\cat C_0=\Db(\coh X)$, such $\kappa$ are called \emph{dualizing complexes}.

In Sections~\ref{se:Groth} and~\ref{se:rel-cpt}, we follow this approach to Grothendieck duality in our abstract setting, with an emphasis on the trichotomy of Corollary~\ref{cor:tricho}.
Let $f^*$ be a functor as in our basic Hypothesis~\ref{hyp:base}. As before, we write $\dual_\kappa= \ihom(-,\kappa)$ for the $\kappa$-twisted duality functor, for any object~$\kappa$. We prove:

\begin{enumerate}[(1)]
\item In the general situation, $f_*$ always intertwines dualities: we have
\begin{align*}
\Dk \circ f_* \cong f_* \circ \Dkk
\end{align*}
where $\kappa\in \cat D$ is any object and $\kappa':=f\uu{1}(\kappa)\in\cat C$; see~Theorem~\ref{thm:duality_preserving}.
\smallbreak
\item Assume that $f^*:\cat D\to \cat C$ satisfies Grothendieck-Neeman duality (Theorem~\ref{thm:GN-intro}).
Let $\cat D_0\subset \cat D$ be a subcategory admitting a dualizing object $\kappa\in \cat D_0$, in which case $\kappa$ induces an equivalence
$\dual_\kappa: \cat D_0^{\op}\stackrel{\sim}{\to} \cat D_0$.
Provided $\cat D_0$ is a $\cat D^c$-submodule (meaning $\cat D^c\otimes \cat D_0\subseteq \cat D_0$), the object $\kappa'= f\uu{1}(\kappa)\cong \DO\otimes f^*(\kappa)$ is dualizing for the following subcategory of~$\cat C$:
\[
\cat C_0:=\{x\in \cat C\mid f_*(c\otimes x)\in \cat D_0 \textrm{ for all } c\in \cat C^c\},
\]
that is, $\kappa'\in \cat C_0$ and
$\dual_{\kappa'}: \cat C_0^{\op}\stackrel{\sim}{\to} \cat C_0$.
Thus, by the formula in~(1), $f_*: \cat C_0\to \cat D_0$ is a morphism of categories with duality.
See Theorem~\ref{thm:rel-Groth} for details, and see Theorem~\ref{Thm:rel-rel-Groth} for a more general relative version.
\smallbreak
\item
Assume moreover that we have the Wirthm\"uller isomorphism of Theorem~\ref{thm:Wirthmueller}.
Because of the monoidal adjunction $f^*:\cat D\adjto \cat C: f_*$, we may consider $\cat C$ as an enriched category over $\cat D$, \ie we may equip $\cat C$ with Hom objects $\underline{\cat C}(x,y):= f_*\,\ihomcat{C}(x,y)$ in~$\cat D$. Then the equivalence $(-)\otimes \DO : \cat C^c \stackrel{\sim}{\to} \cat C^c$ behaves like a \emph{Serre functor relative to $\cat D$}, meaning that there is a natural isomorphism
\begin{align*}
\dual\Crich(x,y) \cong \Crich (y , x \otimes \DO)
\end{align*}
for all $x,y\in \cat C^c$, where $\dual=\ihomcat{D}(-,\unit)$ is the plain duality of~$\cat D$.
If $\cat D=\Der(\kk)$ is the derived category of a field~$\kk$, this reduces to an ordinary Serre functor on the $\kk$-linear category~$\cat C^c$.
See Theorem~\ref{Thm:Serre_duality} and Corollary~\ref{cor:Serre_duality_over_k}.
\end{enumerate}

In algebraic geometry,  we prove that if $X$ is a projective scheme over a regular noetherian base then  the category $\cat C_0$ of~(2) specializes to $\Db(\coh X)$; see Theorem~\ref{thm:cpt-pb-AG}. Thus in this case the results in (1) and (2) specialize to the classical algebro-geometric Grothendieck duality. Similarly, (3) specializes to the classical Serre duality for smooth projective varieties (cf.~Example~\ref{Exa:proj_var}).
 But of course now these results apply more generally, for instance in representation theory, equivariant stable homotopy, and so on, \emph{ad libitum}.

\begin{center}
*\ *\ *
\end{center}

\subsection*{Further examples}

Let us illustrate the broad reach of our setup with two additional examples, now taken from algebra and topology.

Still consider a tensor-exact functor $f^*: \cat D\to \cat C$ satisfying Hypothesis~\ref{hyp:base} and the adjoints $f^*\dashv f_* \dashv f\uu{1}$.  Instead of starting with a subcategory $\cat D_0 \subset \cat D$ as we did above, we may reverse direction and consider a subcategory $\cat C_0\subset\cat C$ with dualizing object~$\kappa'$ and ask under what circumstances may we ``push" the subcategory with duality $(\cat C_0,\kappa')$ along $f_*:\cat C\to \cat D$ to obtain a subcategory with duality in~$\cat D$.

We prove that if $\kappa'$ admits a ``Matlis lift'', that is, an object $\kappa\in \cat D$ such that $f\uu{1}(\kappa)\cong \kappa'$, then $\kappa$ is dualizing for the thick subcategory of $\cat D$ generated by $f_*(\cat C_0)$; see Theorem~\ref{Thm:abstract_DGI_duality}.  This result specializes to classical Matlis duality for commutative noetherian local rings; see Example~\ref{Exa:matlis}.  (For this we now allow dualizing objects to be external, \ie we only assume that $\kappa'\in \cat C$ induces an equivalence $\dual_{\kappa'}: \cat C_0\stackrel{\sim}{\to} \cat C_0$ while possibly $\kappa' \not\in \cat C_0$. Indeed, even when $\kappa'\in \cat C_0$, the lift $\kappa\in \cat D$ need not be in~$\cat D_0$; see Example~\ref{Exa:pontryagin}.) Dwyer, Greenlees and Iyengar \cite{DwyerGreenleesIyengar06} have developed a rich framework which captures several dualities in the style of Pontryagin-Matlis duality, and we show how this connects with our theory in Example~\ref{Exa:DGI}.

Finally, we conclude our article by showing that Neeman's improved version \cite{Neeman92} of Brown-Comenetz duality~\cite{BrownComenetz76} can also be expressed in our framework: It is given by the $\DO$-twisted duality $\dual_{\DO}$ for a certain tensor-exact functor $f^*$ satisfying our basic hypothesis; see Theorem~\ref{thm:NBCduality}. Interestingly, it is possible to show that this functor~$f^*$ is \emph{not} induced by any underlying map~$f$; see Remark~\ref{Rem:no_f}.

\bigbreak
\section{Brown representability and the three basic functors}
\label{se:Brown}%

We begin by recollecting some well-known definitions and results.

Perhaps the most basic fact about adjoints of exact functors on triangulated categories is that they are automatically exact; see~\cite[Lemma 5.3.6]{Neeman01}.

A triangulated category $\cat T$ is said to be \emph{compactly generated} if it admits arbitrary coproducts, and if there exists a set of compact objects $\cat G\subset \cat T$ such that $\cat T(\cat G, t)=0$ implies $ t=0$ for any $t\in \cat T$. An object $t\in \cat T$ is \emph{compact} (a.k.a.\ \emph{finite}) if the functor $\cat T(t,-): \cat T\to \Ab$ sends coproducts in~$\cat T$ to coproducts of abelian groups. We denote by $\cat T^c$ the thick subcategory of compact objects of~$\cat T$. A (contravariant) functor $\cat T\to\cat A$ to an abelian category is called \emph{(co)homological} if it sends exact triangles to exact sequences. The notion of a compactly generated category is extremely useful, thanks to the following result of~Neeman:

\begin{Thm}[Brown representability; see \cite{Neeman96,Krause02}]
\label{thm:Brown}%
Let $\cat T$ be a compactly generated triangulated category.
Then:
\begin{enumerate}[\indent\rm(a)]
\item\label{it:Brown}
A cohomological functor $\cat T\op \to \Ab$ is representable~--- i.e., is isomorphic to one of the form $\cat T(-,t)$ for some $t\in \cat T$ --- if and only if it sends coproducts in~$\cat T$ to products of abelian groups.
\item\label{it:coBrown}
A homological functor $\cat T \to \Ab$ is corepresentable~--- i.e., is isomorphic to one of the form $\cat T(t,-)$ for some $t\in \cat T$ --- if and only if it sends products in~$\cat T$ to products of abelian groups.
\end{enumerate}
\end{Thm}

\begin{Rem}
Theorem~\ref{thm:Brown}\,\eqref{it:Brown} already implies that $\cat T$ admits products (apply it to the functor $\prod_{i}\cat T(-,t_i)$). In turn, this allows for ``dual'' statements, such as~\eqref{it:coBrown}.
\end{Rem}%

\begin{Cor} \label{cor:adjoints}
Let $F: \cat T\to \cat S$ be an exact functor between triangulated categories, and assume that~$\cat T$ is compactly generated. Then:
\begin{enumerate}[\indent\rm(a)]
\item
\label{it:right-Brown}%
$F$ admits a right adjoint if and only if it preserves coproducts.
\item
\label{it:left-Brown}%
$F$ admits a left adjoint if and only if it preserves products.
\end{enumerate}
\end{Cor}

\begin{proof}
As $F$ is exact, the functors $\cat S(F(-), s): \cat T\op\to \Ab$ and $\cat S(s, F(-)):\cat T\to \Ab$ are (co)homological for each $s\in \cat S$, so we can feed them to Theorem~\ref{thm:Brown}.
\end{proof}

\begin{Rem}
\label{rem:comp-gen}%
For $\cat T$ compactly generated, in order to show that a natural transformation $\alpha:F\to F'$ between two coproduct-preserving exact functors $F,F':\cat T\to \cat S$ is an isomorphism, it suffices to prove so for the components $\alpha_x$ at $x\in \cat T^c$ compact. In some cases, this involves giving an alternative definition of $\alpha_x$, valid for $x$ compact, and showing by direct computation that the two definitions coincide. Such computations can become rather involved. We shall leave the easiest of these verifications to the reader but sketch the most difficult ones, hopefully to the benefit of the careful reader.
\end{Rem}

We will also make frequent use of the following two general facts about adjoints on compactly generated categories.

\begin{Prop}[{\cite[Thm.\,5.1]{Neeman96}}] \label{prop:cpt_coprod}
Let $F: \cat S\adjto \cat T:G$ be an adjoint pair of exact functors between triangulated categories $\cat S$ and~$\cat T$, and assume~$\cat S$ compactly generated.
Then~$F$ preserves compact objects iff~$G$ preserves coproducts.
\qed
\end{Prop}

The second general fact will play a crucial role in this paper:

\begin{Lem} \label{Lem:trick_with_adjoints}
Let $F: \cat S\adjto \cat T: G$ be an adjoint pair of exact functors between triangulated categories. Assume $\cat S$ compactly generated and that $F$ preserves compacts.
\begin{enumerate}[\indent\rm(a)]
\item
\label{it:right-trick}%
If the restriction $F|_{\cat S^c}:\cat S^c\to \cat T^c$ admits a right adjoint~$G_0$, then $G$ preserves compacts, and its restriction to compacts is isomorphic to~$G_0$.
\item
\label{it:left-trick}%
If the restriction $F|_{\cat S^c}:\cat S^c\to \cat T^c$ admits a left adjoint~$E_0$ and if $\cat T$ is compactly generated, then $F$ preserves products.
\end{enumerate}
\end{Lem}

\begin{proof}
For every compact $t\in\cat T^c$ and every compact $s\in \cat S^c$, we have a natural bijection
$\cat{S}^c(s,G_0(t))\cong\cat{T}^c(F|_{\cat S^c}(s),t) = \cat T(F(s),t)\cong \cat{S}(s,G(t))$.
By plugging $s:=G_0(t)$, the identity map of $G_0(t)$ corresponds to a certain morphism $\gamma_t:G_0(t)\to G(t)$.
Varying $t\in \cat T^c$, we obtain a natural morphism $\gamma: G_0 \to G|_{\cat T^c}$ by the naturality in~$t$ of the bijection.
By its naturality in~$s$, it actually follows that the bijection is obtained by composing maps $f\in \cat S(s, G_0(t))$ with~$\gamma_t$.
In particular, for any fixed $t\in \cat T^c$ the induced map $\cat S(-,\gamma_t): \cat S(-, G_0(t))\to \cat S(- , G(t))$ is invertible on all $s\in \cat S^c$ by construction, and since $\cat S$ is compactly generated, it is therefore invertible on all~$s\in \cat S$ (cf.~Remark~\ref{rem:comp-gen}). It follows by Yoneda that $\gamma_t$ is an isomorphism. Hence $G(t)\simeq G_0(t)\in \cat S^c$ for every $t\in \cat T^c$, which gives~\eqref{it:right-trick}.

For~\eqref{it:left-trick}, let $\eta:\Id_{\cat T^c}\to F\circ E_0$ be the unit of the adjunction $E_0\adj F|_{\cat S^c}$. For every $x\in \cat T^c$ compact and $s\in\cat S$ arbitrary, we can consider the morphism
$$
\alpha_{x,s}:\cat S(E_0(x),s)\otoo{F}\cat T(FE_0(x),F(s))\otoo{\eta^*}\cat T(x,F(s)).
$$
It is an isomorphism when $s\in \cat S^c$, by the adjunction. Both functors $\cat S(E_0(x),-)$ and $\cat T(x,F(-))$ are homological $\cat S\to \Ab$ and preserve coproducts because $F$ does (it has a right adjoint) and because $x$ and $E_0(x)$ are compact. By Remark~\ref{rem:comp-gen}, $\alpha_{x,s}$ is an isomorphism for every $x\in\cat T^c$ and every~$s\in\cat S$. This kind of ``partial adjoint" suffices to prove that $F$ preserves products, as usual\,: Let $\{s_i\}_{i\in I}$ be a set of objects of~$\cat S$ and $x\in \cat T^c$ be compact and consider the isomorphism
\[
\cat T(x,F(\mathop{\textstyle\prod}_{i\in I} s_i)) \underset{\alpha}{\cong} \cat S(E_0(x),\mathop{\textstyle\prod}_i s_i) \cong \mathop{\textstyle\prod}_i \cat S(E_0(x),s_i) \underset{\alpha}{\cong} \mathop{\textstyle\prod}_i \cat T(x,F s_i)
\cong \cat T(x,\mathop{\textstyle\prod}_{i\in I} Fs_i)\,.
\]
One verifies that this is the morphism induced by the canonical map $F(\prod_{i\in I} s_i)\to \prod_{i\in I}F(s_i)$ and since~$\cat T$ is compactly generated, this map is an isomorphism.
\end{proof}

\begin{center}
*\ *\ *
\end{center}

We now let the tensor~$\otimes$ enter the game.

\begin{Def} \label{Def:rigidly}
A tensor-triangulated category $\cat C$ (\ie~a triangulated category with a compatible closed symmetric monoidal structure, see~\cite[App.~A.2]{HoveyPalmieriStrickland97})
is called \emph{rigidly-compactly} generated if it is compactly generated and if compact objects and rigid objects coincide; in particular, the tensor unit object~$\unit$ is compact. We denote the tensor by $\otimes:\cat C\times \cat C\too \cat C$ and its right adjoint by $\ihom:\cat C\op\times \cat C\too \cat C$ (internal Hom). An object~$x$ is \emph{rigid} if the natural map $\ihom(x,\unit)\otimes y\to \ihom(x,y)$ is an isomorphism for all~$y$. Rigid objects are often called ``(strongly) dualizable'' in the literature but we avoid this terminology to prevent any possible confusion with our ``dualizing objects''.
\end{Def}

\begin{Rem} \label{Rem:selfduality_on_cpts}
When $\cat C$ is rigidly-compactly generated, its subcategory of compact objects $\cat C^c\subset \cat C$ is a thick subcategory, closed under~$\otimes$. It admits the canonical duality $\dual=\ihom(-,\unit):(\cat C^c)\op\to \cat C^c$ satisfying~$\dual^2\cong\Id$. See details in~\cite[App.~A]{HoveyPalmieriStrickland97} for instance, where our rigidly-compactly generated tensor-triangulated categories are called ``unital algebraic stable homotopy categories".
\end{Rem}

Let us mention at this point a few important examples of rigidly-compactly generated categories~$\cat C$ arising in various fields of mathematics.

\begin{Exa}
\label{exa:schemes}%
Let $X$ be a quasi-compact and quasi-separated scheme. Let $\cat C:= \Der_{\Qcoh}(X)$ be the derived category of complexes of $\mathcal O_X$-modules having quasi-coherent homology (see~\cite{Lipman09}). It is rigidly-compactly generated, and its compact objects are precisely the perfect complexes:
 $(\Der_{\Qcoh}(X))^c=\Dperf(X)$ (see \cite{BondalVandenbergh03}). The latter are easily seen to be rigid for the derived tensor product~$\otimes= \otimes^{\mathrm L}_{\mathcal O_X}$.
If moreover $X$ is separated, there is an equivalence $\Der_{\Qcoh}(X)\simeq \Der(\Qcoh X)$ with the derived category of complexes of quasi-coherent $\mathcal O_X$-modules (see \cite{BoekstedtNeeman93}).
If $X=\Spec(A)$ is affine, then $\Der(\Qcoh X)\simeq \Der(A\MMod)$ with compacts $\Der(A\MMod)^c\simeq \Khocat^\mathrm{b}(A\pproj)$, the homotopy category of bounded complexes of finitely generated projectives.
\end{Exa}

\begin{Exa}
\label{exa:SH}%
Let $G$ be a compact Lie group. Then $\cat C := \SH(G)$, the homotopy category of ``genuine'' $G$-spectra indexed on a complete $G$-universe (see \cite[\S9.4]{HoveyPalmieriStrickland97}), is rigidly-compactly generated. The suspension $G$-spectra $\Sigma^\infty_+G/H$, with~$H$ running through all closed subgroups of~$G$, form a set of rigid-compact generators which includes the tensor unit $\unit= \Sigma^\infty_+G/G$.
\end{Exa}

\begin{Exa}
\label{exa:Stab}%
Let $G$ be a finite group and let $\kk$ be a field. Then $\cat C:=\Stab(\kk G)$, the stable category of $\kk G$-modules modulo projectives, is rigidly-compactly generated. (Note that the derived category $\Der(\kk G)$, though compactly generated, is not \emph{rigidly}-compactly generated because its unit $\unit=\kk$ is not compact).
More generally, $G$ could be a finite group scheme over $\kk$ (see e.g.~\cite[Theorem~9.6.3]{HoveyPalmieriStrickland97}).
\end{Exa}

\begin{Exa}
\label{exa:SHk}%
Let $\kk$ be a field and let $\cat C := \SHA(\kk)$ denote the stable $\bbA^1$-homotopy category.
Twists of smooth projective \mbox{$\kk$-varieties} are rigid-compact in~$\cat C$. They generate the whole category under resolution of singularities (see \cite{Riou05}). Hence if $\kk$ has characteristic zero, $\SHA(\kk)$ is rigidly-compactly generated.
\end{Exa}

\begin{Exa}
\label{exa:BraveNew}%
Let $A$ be a ``Brave New'' commutative ring, that is, a structured commutative ring spectrum.
To fix ideas, we can understand $A$ to be a commutative \mbox{$\Sphere$-algebra} in the sense of \cite{EKMM97}.
Then its derived category $\Der(A)$, \ie the homotopy category of $A$-modules, is a rigidly-compactly generated category, which is generated by its tensor unit~$A$ (see e.g.~\cite[Example 1.2.3(f)]{HoveyPalmieriStrickland97} and \cite[Example 2.3(ii)]{SchwedeShipley03}).
For example, every commutative dg-ring has an associated commutative $\Sphere$-algebra (its Eilenberg-MacLane spectrum) whose derived category is equivalent, as a tensor triangulated category, to the derived category of dg-modules (see~\cite{Shipley07} and \cite[Theorem 5.1.6]{SchwedeShipley03}).  Thus, derived categories of commutative dg-rings are also rigidly-compactly generated.
\end{Exa}

\begin{Cor}
\label{cor:base}%
Let $\cat C$ and $\cat D$ be rigidly-compactly generated categories, and let $f^*: \cat D\to \cat C$ be as in our basic Hypothesis~\ref{hyp:base}. Then $f^*$ preserves compacts and admits a right adjoint $f_*:\cat C\to \cat D$, which itself admits a right adjoint $f\uu{1}:\cat D\to \cat C$.
\end{Cor}

\begin{proof}
Since $f^*$ preserves coproducts by assumption, $f_*$ exists by Brown Representability, Cor.\,\ref{cor:adjoints}\,\eqref{it:right-Brown}. Since $f^*$ is symmetric monoidal by assumption, it must send rigid objects of~$\cat D$ to rigid objects of~$\cat C$ (see e.g.\ \cite[\S III.1]{LMS86}). Hence it must preserve compacts ($=$~rigids). By Proposition~\ref{prop:cpt_coprod}, $f_*$ preserves coproducts and we can apply another layer of Brown Representability to~$f_*$ in order to get~$f\uu{1}$.
\end{proof}

Our three functors $f^*\dashv f_*\dashv f\uu{1}$ automatically satisfy some basic formulas.

\begin{Prop} \label{Prop:right_proj_formula}
	Let $f^*\dashv f_*\dashv f\uu{1}$ be as in Corollary~\ref{cor:base}.
Then there is a canonical natural isomorphism
\begin{align}
\pi : x \otimes f_*(y) &\overset{\sim}{\longrightarrow} f_*(f^*(x)\otimes y) \label{eq:right_proj_formula}
\end{align}
for all $x\in \cat D$ and $y\in \cat C$, obtained from $f^*(x \otimes f_*(y))\cong f^*(x)\otimes f^*f_*(y)\to f^*(x)\otimes y$ by adjunction.
We also have three further canonical isomorphisms as follows:
\begin{align}
\ihomcat{D}(x, f_*y) &\cong f_*\, \ihomcat{C}(f^* x, y)  \label{eq:inthom-f^*-f_*} \\
\ihomcat{D}(f_* x, y) &\cong f_*\, \ihomcat{C}(x,f\uu{1} y)  \label{eq:inthom-f_*-f^1} \\
f\uu{1} \ihomcat{D}(x,y) &\cong \ihomcat{C}(f^*x, f\uu{1}y) \label{eq:co-RPF} \,.
\end{align}
\end{Prop}

\begin{Ter}
We call~\eqref{eq:right_proj_formula} the \emph{(right) projection formula}. Equations~\eqref{eq:inthom-f^*-f_*} and~\eqref{eq:inthom-f_*-f^1} are \emph{internal realizations} of the two adjunctions $f^*\dashv f_*\dashv f\uu{1}$, from which the adjunctions can be recovered by applying $\ihomcat{D}(\unit_\cat{D},-)$.
	Note that~\eqref{eq:inthom-f_*-f^1} specializes to~\eqref{eq:i-DO} by inserting $y=\unit_\cat D$.
\end{Ter}

\begin{proof}
	The map~$\pi$ is clearly well-defined for all~$x$ and~$y$ and is automatically invertible whenever~$x$ is rigid (cf.~\cite[Prop.\,3.2]{FauskHuMay03}).
Fixing an arbitrary~$y\in \cat C$, note that both sides of~\eqref{eq:right_proj_formula} are exact and commute with coproducts in the variable~$x$.
As $\cat D$ is generated by its compact ($=$~rigid) objects, $\pi$ is an isomorphism for all~$x\in \cat D$ (Rem.\,\ref{rem:comp-gen}). This proves the first isomorphism, \ie the projection formula.

Now we can derive from it two of the other equations by taking adjoints. (Recall that if $F_i\adj G_i$ for $i=1,2$ then $F_1F_2\adj G_2G_1$. Note the order-reversal.) First, by fixing~$x$ we see two composite adjunctions
\begin{align*}
x\otimes f_*
= (x\otimes -)\circ f_*
\quad & \dashv \quad
f\uu{1} \circ \ihomcat{D}(x,-)
\\
f_*(f^*(x)\otimes -)
= f_*\circ (f^*(x)\otimes-)
\; & \dashv \;
\ihomcat{C}(f^*x, -)\circ f\uu{1}
= \ihomcat{C}(f^*x, f\uu{1} (-)) \,.
\end{align*}
Since $\pi$ is an isomorphism of the left adjoints, by the uniqueness of right adjoints it induces an isomorphism between the right ones, \ie we get~\eqref{eq:co-RPF}. (The naturality in~$x$ is guaranteed by the fact that the two adjunctions above are actually natural families of adjunctions parametrized by~$x$.)
If we fix~$y$ instead, we get adjunctions
\begin{align*}
(-)\otimes f_*(y)
\quad & \dashv\quad
\ihomcat{D} (f_*y, -)
\\
f_*(f^*(-)\otimes y)
= f_*\circ (-\otimes y) \circ f^*
\; & \dashv\;
f_*\circ \ihomcat{C}(y,-)\circ f\uu{1}
= f_*\,\ihomcat{C}(y, f\uu{1}(-))
\end{align*}
from which we derive the natural isomorphism~\eqref{eq:inthom-f_*-f^1}.
By fixing $x$ in the isomorphism $f^*(x)\otimes f^*(y)\cong f^*(x\otimes y)$ given by the monoidal structure of~$f^*$, we obtain
\begin{align*}
f^*(x) \otimes f^*
= (f^*(x)\otimes -)\circ f^*
\quad & \dashv \quad
f_*\, \ihomcat{C}(f^*x, -)
\\
f^*(x\otimes -) = f^*\circ (x\otimes-)
\quad & \dashv\quad
\ihomcat{D}(x,-)\circ f_*
= \ihomcat{D}(x,f_*(-))
\end{align*}
from which we derive the remaining relation~\eqref{eq:inthom-f^*-f_*}.
\end{proof}

\begin{Rem} \label{Rem:conjugate_formulas}
The reasoning of the previous proof will be used several times, so it is worth spending a little thought on it.  Let's say we have some \emph{formula}, by which we mean a natural isomorphism $F_1\circ \ldots \circ F_n \cong F'_1\circ\ldots \circ F'_m$ between composite functors, in which every factor is part of an adjunction $F_i\dashv G_i$ and $F'_j\dashv G'_j$.  By taking right adjoints on both sides we derive a formula $G_n G_{n-1}\cdots G_1 \cong G'_m G'_{m-1} \cdots G'_1$. Actually the two formulas are equivalent, since we may recover the first one by taking left adjoints in the second one.  Following \cite{FauskHuMay03}, we can say that the two formulas are \emph{conjugate}, or \emph{adjunct}.  Note however that if the original formula admits two different factor-decompositions as above, we would obtain a different conjugate formula from each choice of decomposition.  This is illustrated by the previous proposition, in which~\eqref{eq:inthom-f_*-f^1} and~\eqref{eq:co-RPF} are obtained from two different decompositions of~\eqref{eq:right_proj_formula}.  In this case,~\eqref{eq:right_proj_formula} is a formula between functors of two variables~$x$ and~$y$, and the two decompositions have been obtained by first fixing either~$x$ or~$y$.  Note that the tensor formula $f^*(x\otimes y)\cong f^*(x)\otimes f^*(y)$ is symmetric in $x$ and~$y$, hence the two resulting decompositions yield the same conjugate formula~\eqref{eq:inthom-f^*-f_*}.  All our conjugate formulas will come in such couplets or triplets and will be obtained in this way from a starting formula in either one or two variables.  The systematic exploitation of this principle will greatly simplify the search for new relations.  When repeating this reasoning below we will mostly leave the straightforward details to the reader.
\end{Rem}

\bigbreak
\section{Grothendieck-Neeman duality and ur-Wirthm\"uller}
\label{se:GN}%

We want to prove Theorem~\ref{thm:GN-intro}, for which we need some preparation. Recall the basic set-up as in Hypothesis~\ref{hyp:base} and the three functors $f^*\adj f_*\adj f\uu{1}$ (Cor.\,\ref{cor:base}). We first focus on the new, slightly surprising facts. The following lemma should be compared to the well-known property presented in Proposition~\ref{prop:cpt_coprod}.

\begin{Lem} \label{Lem:f_*-cpts}
If $f^*:\cat D\to \cat C$ has a left adjoint $f\dd{1}\adj f^*$, \ie if $f^*$ preserves products, then its right adjoint $f_*$ preserves compact objects: $f_*(\cat C^c)\subseteq \cat D^c$.
\end{Lem}

\begin{proof}
Recall that~$f^*$ preserves coproducts by our standing hypothesis, hence~$f\dd{1}$ preserves compacts (Prop.\,\ref{prop:cpt_coprod}). Therefore $f\dd{1}\adj f^*$ restricts to an adjunction $f\dd{1}: \cat C^c \adjto \cat D^c: f^*$ on compact objects. Since compacts are rigid, duality provides equivalences of (tensor) categories $\dual:=\ihomcat{C}(- , \unit): (\cat C^c)\op \to \cat C^c$ and $\dual:=\ihomcat{D}(- , \unit): (\cat D^c)\op \to \cat D^c$ which are quasi-inverse to themselves (\ie\ $\dual^{-1}=\dual\op$).
Moreover, the symmetric monoidal functor~$f^*$ preserves rigid objects~$c$ and their tensor duals~$\dual(c)$ (cf.~\cite[\S III.1]{LMS86}), so that the following square commutes (up to isomorphism of functors):
\[
\xymatrix{
(\cat D^c)\op \ar[d]_{(f^*)\op} \ar[r]^-\dual_-\sim & \cat D^c \ar[d]^{f^*} \\
(\cat C^c)\op \ar[r]^-\dual_-\sim & \cat C^c.
}
\]
This self-duality implies that the composite functor $f_*^c:=\dual\circ (f\dd{1})\op\circ  \dual^{-1} = \dual f\dd{1} \dual \colon \cat C^c\to \cat D^c$ is \emph{right} adjoint to~$f^*\colon \cat D^c \to \cat C^c$.
By Lemma~\ref{Lem:trick_with_adjoints}\,\eqref{it:right-trick} (applied to~$F:=f^*$), the right adjoint $f_*$ to $f^*$ must preserve compact objects.
\end{proof}

\begin{Prop} \label{Prop:midstep}
Suppose that $f_*:\cat C\to \cat D$ preserves compacts. Then there is a canonical natural isomorphism
$ f\uu{1}(x) \otimes f^*(y) \stackrel{\sim}{\longrightarrow} f\uu{1}(x\otimes y)$
for all $x,y\in \cat D$.
\end{Prop}

\begin{proof}
The natural comparison map $f\uu{1}(x) \otimes f^*(y) \to f\uu{1}(x\otimes y)$ can always be constructed out of the counit $\eps : f_*f\uu{1}\to \Idcat{C}$ of the adjunction $f_*\dashv f\uu{1}$ as follows:
\begin{align*}
\eps_x \otimes \id_y
&\;\; \in \;\;  \cat D (f_*f\uu{1}(x) \otimes y, x\otimes y) &&  \\
&\;\; \cong \;\;  \cat D(f_*(f\uu{1}(x) \otimes f^*(y) ), x\otimes y) && \textrm{projection formula Prop.\,\ref{Prop:right_proj_formula}} \\
&\;\; \cong \;\;  \cat C  (f\uu{1}(x) \otimes f^*(y) , f\uu{1}(x\otimes y)) && \textrm{adjunction } f_*\dashv f\uu{1} \,.
\end{align*}
If $y\in \cat D^c$ is rigid, we have for all $z\in \cat C$ a natural isomorphism:
\begin{align*}
\cat C(z, f\uu{1}(x) \otimes f^*(y))
&\;\;\cong\;\;  \cat C(z\otimes \dual f^*(y) , f\uu{1}(x))&& f^*(y) \textrm{ is rigid} \\
&\;\;\cong\;\;  \cat C(z\otimes f^*\dual(y) , f\uu{1}(x)) && f^*\dual \cong \dual f^*\textrm{ on rigids} \\
&\;\;\cong\;\;  \cat D(f_*(z\otimes f^*\dual(y)) , x) && \textrm{adjunction } f_*\dashv f\uu{1} \\
&\;\;\cong\;\; \cat D(f_*(z)\otimes \dual(y) , x) && \textrm{projection formula Prop.\,\ref{Prop:right_proj_formula}} \\
&\;\;\cong\;\;  \cat D(f_*(z), x\otimes y) && y \textrm{ is rigid} \\
&\;\;\cong\;\; \cat C (z, f\uu{1}(x\otimes y)) && \textrm{adjunction } f_*\dashv f\uu{1}.
\end{align*}
A tedious but straightforward diagram chase verifies that this isomorphism is merely post-composition by the general comparison map
	$f\uu{1}(x) \otimes f^*(y) \to f\uu{1}(x \otimes y)$
	previously defined.
	Hence, by Yoneda, we conclude that the general comparison map is an isomorphism whenever $y$ is rigid.
By Proposition~\ref{prop:cpt_coprod}, the hypothesis on~$f_*$ is equivalent to $f\uu{1}$ preserving coproducts.
Hence both sides of the comparison map $f\uu{1}(x) \otimes f^*(y) \to f\uu{1}(x\otimes y)$ are coproduct-preserving exact functors in both variables. Hence this comparison map is invertible for all $x,y\in \cat D$ (Remark~\ref{rem:comp-gen}).
\end{proof}

We are now ready to prove our generalized Grothendieck-Neeman duality theorem.
Recall from Definition~\ref{def:DO} that $\DO:= f\uu{1}(\unit)\in \cat C$ is the \emph{relative dualizing object} associated with the given functor $f^*:\cat D\to \cat C$.

\begin{Thm}
\label{thm:GN}%
Let $f^*: \cat D\to \cat C$ be as in our basic Hypothesis~\ref{hyp:base} and consider the adjoints $f^*\adj f_*\adj f\uu{1}$ (Cor.\,\ref{cor:base}).
Then the following conditions are equivalent:
\begin{enumerate}[\indent\rm(a)]
\item\label{it:a1}
The functor $f^*$ admits a left adjoint $f\dd{1}$.
\item\label{it:b1}
The functor $f^*$ preserves products.
\item\label{it:c1}
The functor $f\uu{1}$ admits a right adjoint~$f\dd{-1}$.
\item\label{it:d1}
The functor $f\uu{1}$ preserves coproducts.
\item\label{it:e1}
The functor $f_*$ preserves compact objects.
\end{enumerate}
Furthermore, if~\eqref{it:a1}-\eqref{it:e1} hold true then $f\dd{1}\adj f^*\adj f_*\adj f\uu{1}\adj f\dd{-1}$ satisfy the following additional relations given by canonical natural isomorphisms:
\begin{flalign}
&\kern5em & \Aboxed{f\uu{1} & \cong \DO \otimes f^*(-)} \label{eq:Groth}& \textrm{(Grothendieck duality)}
\\
&& f\dd{-1} & \cong f_*\,\ihomcat{C}(\DO,-) \label{eq:co-Groth}
\\[1em]
&& f\uu{1}(x \otimes y) & \cong f\uu{1}(x) \otimes f^*(y) \label{eq:gen_Groth} && \\
&& \ihomcat{D}(x ,  f\dd{-1} y) &\cong f_* \, \ihomcat{C}(f\uu{1} x, y) \label{eq:gen_co-Groth}
  && \\
&& \ihomcat{D}(x,f\dd{-1}y) &\cong f\dd{-1} \,\ihomcat{C}(f^*x, y) \label{eq:pseudo-inthom} &&
\end{flalign}
\begin{flalign}
&&\kern10.1em f^*(-) & \cong \ihomcat{C}(\DO, f\uu{1}(-)) \label{eq:co-Wirth} &&
\\
&& \Aboxed{\ f\dd{1}(-) & \cong f_*(\DO\otimes -) \vphantom{f^f_f}\ }  \label{eq:ur-Wirth1} & \textrm{(ur-Wirthm\"uller)}
\\[1em]
&& \Aboxed{\ x \otimes f\dd{1}(y) & \cong f\dd{1}(f^*(x)\otimes y) \vphantom{f^f_f}\ } \label{eq:left_proj_formula} & \textrm{(left projection formula)}
\\
&& f^*\ihomcat{D}(x,y)  & \cong \ihomcat{C}(f^*x, f^*y)  \label{eq:f*-closed} \\
&& \ihomcat{D}(f\dd{1}x, y) &\cong  f_*\, \ihomcat{C} (x , f^* y). \label{eq:co-left_proj_formula}
\end{flalign}
\end{Thm}

\begin{Rem}
The existence of \emph{any} natural isomorphism as in Grothendieck duality~\eqref{eq:Groth} implies that $f\uu{1}$ preserves coproducts (\ie property~\eqref{it:d1} holds). Hence~\eqref{eq:Groth} is not only a consequence of, but is equivalent to, conditions~\eqref{it:a1}-\eqref{it:e1} of the theorem. Similarly, the more general~\eqref{eq:gen_Groth} is also equivalent to~\eqref{it:a1}-\eqref{it:e1}. Finally, if there exists \emph{any} isomorphism as in~\eqref{eq:co-Wirth}, then $f^*$ must preserve products, since so do the left adjoints $\ihomcat{C}(\DO,-)$ and $f\uu{1}$. Hence~\eqref{eq:co-Wirth} is also an equivalent condition for Theorem~\ref{thm:GN} to hold. We note this for completeness but it is unlikely that such isomorphisms can be established in practice \emph{before} any of~\eqref{it:a1}-\eqref{it:e1} is known.
\end{Rem}

\begin{Rem} \label{Rem:conj_sets_in_GNthm}
We will see in the proof that each group of equations in~\eqref{eq:Groth}-\eqref{eq:co-left_proj_formula}, as displayed above, forms a conjugate set of formulas in the sense of Remark~\ref{Rem:conjugate_formulas}.
\end{Rem}

\begin{Rem}
All of the adjunctions $f\dd{1}\adj f^*\adj f_*\adj f\uu{1}\adj f\dd{-1}$ now have an internal realization in~$\cat D$ by~\eqref{eq:co-left_proj_formula},~\eqref{eq:inthom-f^*-f_*},~\eqref{eq:inthom-f_*-f^1}, and~\eqref{eq:gen_co-Groth}, respectively.
\end{Rem}

\begin{Rem} \label{Rem:computer}
We can further combine the fundamental formulas of Theorem~\ref{thm:GN}, for instance by composing Grothendieck duality~\eqref{eq:Groth} and the ur-Wirthm\"uller~\eqref{eq:ur-Wirth1} isomorphism and then variating by conjugation:
\begin{flalign*}
 f\dd{1} (f^*(x)\otimes y) &\cong f_*(f\uu{1}(x)\otimes y) 
  \\
 \ihomcat{C}(f^*x,f^*y) &\cong \ihomcat{C}(f\uu{1}x, f\uu{1}y) 
  \\
 f_*\, \ihomcat{C}(x, f^*y) &\cong f\dd{-1} \ihomcat{C}(x, f\uu{1}y)
  \,.
\end{flalign*}
Or we may plug Grothendieck duality into ur-Wirthm\"uller's adjunct~\eqref{eq:co-Wirth} to obtain
\begin{flalign*}
f^* &\cong \ihomcat{C}(\DO,\DO\otimes f^*(-))
\end{flalign*}
which, when applied to the tensor unit, specializes to the important relation
\begin{flalign}
 \unit_{\cat C} &\cong \ihomcat{C}(\DO,\DO) \label{eq:ihom(w,w)1} \,.
\end{flalign}
Further combinations of the original formulas are left to the interested reader.
\end{Rem}

\begin{proof}[Proof of Theorem~\ref{thm:GN}]
We already know that~\eqref{it:a1}\IFF\eqref{it:b1} and~\eqref{it:c1}\IFF\eqref{it:d1}\IFF\eqref{it:e1} by Brown representability (Cor.\,\ref{cor:adjoints}) and by Proposition~\ref{prop:cpt_coprod}. We also isolated the non-obvious parts of the equivalences in Lemma~\ref{Lem:trick_with_adjoints}\,\eqref{it:left-trick} and Lemma~\ref{Lem:f_*-cpts}, which give~\eqref{it:e1}$\then$\eqref{it:b1} and~\eqref{it:a1}$\then$\eqref{it:e1} respectively. So we can assume that~\eqref{it:a1}-\eqref{it:e1} hold true and we now turn to proving Formulas~\eqref{eq:Groth}-\eqref{eq:co-left_proj_formula}.

Proposition~\ref{Prop:midstep} already gives~\eqref{eq:gen_Groth}, which then specializes to~\eqref{eq:Groth} by setting $x:=\unit_\cat D$. We now construct the canonical ur-Wirthm\"uller isomorphism~\eqref{eq:ur-Wirth1}.
For every $x\in \cat C$, consider the composite
\begin{equation}
\label{eq:explicit-ur-W}%
f_*(x \otimes \DO) \xrightarrow{f_*(\eta \otimes 1)}f_*(f^* f\dd{1} x \otimes \DO) \cong f\dd{1}x \otimes f_*\DO =f\dd{1}x \otimes f_*f\uu{1} \unit \xrightarrow{1 \otimes \eps} f\dd{1}x
\end{equation}
where the middle map is the right projection formula~\eqref{eq:right_proj_formula} and $\eta:\Id\to f^*f\dd{1}$ and $\eps:f_*f\uu{1}\to \Id$ are the unit and counit of these  adjunctions. By Remark~\ref{rem:comp-gen}, it suffices to show that this map is an isomorphism for $x\in\cat C^c$ compact, because both ends of~\eqref{eq:explicit-ur-W} preserve coproducts in~$\cat C$ (being composed of left adjoints).
Now, for every $d\in \cat D$ and for $x\in \cat C^c$ rigid we compute:
\begin{align*}
\cat D(d, f_*(x\otimes \DO))
&\;\;\cong\;\; \cat C(f^*(d), x\otimes \DO)&& f^*\dashv f_*\\
&\;\;\cong\;\; \cat C(\dual(x)\otimes f^*(d),  \DO) &&  x\in \cat C^c\textrm{ is rigid}\\
&\;\;\cong\;\; \cat D(f_*(\dual x \otimes f^*d), \unit) && \DO=f\uu{1}(\unit)\ \textrm{and}\ f_*\dashv f\uu{1}\\
&\;\;\cong\;\; \cat D( f_*(\dual x) \otimes d, \unit) && \text{projection formula~\eqref{eq:right_proj_formula} }\\
&\;\;\cong\;\; \cat D(d,\dual f_*\dual(x)) && f_*\dual(x)\in\cat D^c\textrm{ by Lemma~\ref{Lem:f_*-cpts}} \\
&\;\;\cong\;\; \cat D(d, f\dd{1}(x)). &&
\end{align*}
The last isomorphism holds because, thanks to~\eqref{it:e1}, the adjunction $f\dd{1}\adj f^* \adj f_*$ restricts to the categories of compact objects,  so that the dual~$\dual$ intertwines the two restricted functors: $\dual f_*\dual\cong f\dd{1}$ on~$\cat C^c$. By Yoneda, we obtain from the above an isomorphism $f_*(x\otimes \DO)\cong f\dd{1}(x)$ for $x\in\cat C^c$ and we ``only" need to show that it coincides with the canonical map~\eqref{eq:explicit-ur-W}. This is an adventurous diagram chase that we outline in more detail this time, since it might be harder to guess.

Following through the chain of isomorphisms, we can reduce the problem to checking the commutativity of the following diagram:
\begin{equation}\label{eq:dia_urWirth}
\begin{gathered}
\xymatrix{
  f_*(x\otimes \DO) \ar[r]^-{\mathrm{coev}\otimes 1} \ar[d]_-{f_*(\eta\otimes 1)}
& \dual f_*\dual x\otimes f_*\dual x\otimes f_*(x\otimes \DO)
\ar[d]^{1 \otimes \mathrm{lax}} \\
f_*(f^*f\dd{1}x\otimes \DO) \ar[d]_{\pi}
& \dual f_*\dual x\otimes f_*(\dual x\otimes x\otimes \DO) \ar[d]^{1\otimes f_* (\mathrm{ev}\otimes 1)}
\\
f\dd{1}x\otimes f_*\DO \ar[r]^{\cong}
& \dual f_*\dual x\otimes f_*\DO
}
\end{gathered}
\end{equation}
where $\pi$ is the projection formula isomorphism~\eqref{eq:right_proj_formula}. Using an explicit description of the bottom isomorphism $f\dd{1}x\cong \dual f_*\dual x$ in terms of the unit and counit of $f\dd{1} \adj f^*$ and the duality maps, one can check that the composite along the left and bottom edges is equal to
$$
\kern-.2em
\xymatrix@C=6em@R=1em{
f_*(x\otimes \DO)\cong f_*(f^*\unit\otimes x\otimes \DO) \ar[r]^-{f_*(f^*\mathrm{coev}\otimes 1\otimes 1)} & f_*(f^*(\dual f_*\dual x\otimes f_*\dual x)\otimes x\otimes \DO) \ar[d]^-{f_*(f^*\pi\otimes 1)} 
\\ & f_*(f^*f_*(f^*\dual f_*\dual x\otimes \dual x)\otimes x\otimes \DO) \ar[d]^-{f_*(\eps\otimes 1)}
\\ & f_*(f^*\dual f_*\dual x\otimes \dual x\otimes x\otimes \DO) \ar[d]^-{f_*(1\otimes \mathrm{ev}\otimes 1)} 
\\ & f_*(f^*\dual f_*\dual x\otimes \DO) \ar[d]^-{\pi} 
\\ & \dual f_*\dual x\otimes f_*\DO.
}
$$
This composite can then be checked to agree with
$$
\xymatrix@R=1em{
f_*(x\otimes \DO)=\unit\otimes f_*(x\otimes \DO) \ar[r]^-{\mathrm{coev}\otimes 1} & \dual f_*\dual x\otimes f_*\dual x\otimes f_*(x\otimes \DO) \ar[d]^-{\pi\otimes 1}
\\ & f_*(f^*\dual f_*\dual x\otimes \dual x)\otimes f_*(x\otimes \DO) \ar[d]^-{\lax}
\\ & f_*(f^*\dual f_*\dual x\otimes \dual x\otimes x\otimes \DO) \ar[d]^-{f_*(1\otimes \mathrm{ev}\otimes 1)}
\\ & f_*(f^*\dual f_*\dual x\otimes \DO) \ar[d]^-{\pi}
\\ & \dual f_*\dual x\otimes f_*\DO.
}
$$
Using this last description, the commutativity of diagram~\eqref{eq:dia_urWirth} can be established.
In carrying out these verifications, the commutativity of the following diagrams
\begin{equation}
\label{eq:lax-pi-auxiliary}
\!\vcenter{\xymatrix@C=1.5em@R=1em{
  f_*a\otimes f_*b \ar[r]^-{\pi} \ar[rd]_-{\lax}
& f_*(f^*f_*a\otimes b) \ar[d]^-{f_*(\eps\otimes 1)}
\\
& f_*(a\otimes b)
}}
\quad\text{and}\quad
\vcenter{\xymatrix@C=2em@R=1em{
  a\otimes f_*b\otimes f_*c \ar[r]^-{\pi\otimes 1} \ar[d]_-{1\otimes \lax}
& f_*(f^*a\otimes b)\otimes f_*c \ar[d]^-{\lax}
\\
  a\otimes f_*(b\otimes c) \ar[r]^-{\pi}
& f_*(f^*a\otimes b\otimes c)
}}\kern-1em
\end{equation}
will prove to be useful. The remaining details are now left to the careful reader.
\goodbreak

We have now established~\eqref{eq:Groth},~\eqref{eq:gen_Groth} and~\eqref{eq:ur-Wirth1}, from which we derive the other ones by the general method of Remark~\ref{Rem:conjugate_formulas}. Taking right adjoints of the functors in~\eqref{eq:ur-Wirth1} yields~\eqref{eq:co-Wirth}. Taking right adjoints of the functors in~\eqref{eq:gen_Groth} for each fixed~$x$ yields~\eqref{eq:gen_co-Groth}, and taking right adjoints for each fixed~$y$ yields~\eqref{eq:pseudo-inthom}. (The left-hand-sides of the two latter formulas coincide, because \mbox{$f\uu{1}(x\otimes y)$} is symmetric in $x$ and~$y$.) On the other hand, taking right adjoints in~\eqref{eq:Groth} yields~\eqref{eq:co-Groth}. Also,~\eqref{eq:ihom(w,w)1} is~\eqref{eq:co-Wirth} evaluated at~$\unit_{\cat D}$.
The left projection formula~\eqref{eq:left_proj_formula} follows by conjugating the right projection formula~\eqref{eq:right_proj_formula} by the ur-Wirthm\"uller isomorphism~\eqref{eq:ur-Wirth1}:
\[
\xymatrix@R=1.5em{
x \otimes f\dd{1}(y) \ar[r]^-\cong \ar@{..>}[d]_-\cong &
  x \otimes f_*( \DO \otimes y ) \ar[d]^-\cong \\
f\dd{1}(f^*(x)\otimes y) \ar[r]^-\cong &
 f_*(f^*(x) \otimes \DO \otimes y).
}
\]
This new two-variable equation~\eqref{eq:left_proj_formula} has two conjugate formulas~\eqref{eq:f*-closed} and~\eqref{eq:co-left_proj_formula}, obtained by taking right adjoints while fixing~$x$ or~$y$, respectively.
\end{proof}

\begin{Exa}[Algebraic geometry]
\label{exa:alg_geom_general}
Let $f:X\to Y$ be a morphism of quasi-compact and quasi-separated schemes, as in Example~\ref{exa:schemes}, and consider the (derived) inverse image functor
$f^*:\cat D = \Der_{\Qcoh}(Y) \to \Der_{\Qcoh}(X) = \cat C$.
It is easy to see that~$f^*$ satisfies our basic Hypothesis~\ref{hyp:base};
its right adjoint is the derived push-forward $f_*= \Rder f_*$, whose right adjoint~$f\uu{1}$ is the twisted inverse image functor, usually written~$f^\times$ or~$f^!$ (see~\cite{Lipman09}).
Then the functor $f^*$ satisfies Grothendieck-Neeman duality precisely when the morphism~$f$ is \emph{quasi-perfect} \cite[Def.\,1.1]{LipmanNeeman07}. Indeed, the latter means by definition that $\Rder f_*$ preserves perfect complexes, \ie compact objects.
In this context, our Theorem~\ref{thm:GN} recovers the original results of Neeman that have inspired us; see~\cite[Prop.\,2.1]{LipmanNeeman07} for a geometric statement in the same generality as we obtain here by specializing our abstract methods. Yet, even when specialized to algebraic geometry, our theorem is somewhat stronger, because it includes the extra information about the left adjoint~$f\dd{1}$ of~$f^*$, whose existence is equivalent to the quasi-perfection of~$f$ and which is necessarily given by the ur-Wirthm\"uller formula $f\dd{1} \cong \DO \otimes f_*$.

The article \cite{LipmanNeeman07} contains a thorough geometric study of quasi-perfection.
Among other things, it is shown that $f$ is quasi-perfect iff it is proper and of finite tor-dimension.
In particular if $f:X\to Y$ is \emph{finite} then it is quasi-perfect iff $f_*(\unit_\cat C)= \Rder f_*(\mathcal O_X)$ is a perfect complex. See Examples~2.2 of \emph{loc.\,cit.}\ for more.
\end{Exa}

\begin{Exa}[Affine case] \label{Exa:AG_affine}
Let $\phi: B\to A$ be a morphism of commutative rings. We may specialize Example~\ref{exa:alg_geom_general} to \mbox{$f:=\Spec(\phi): X:=\Spec(A)\to \Spec(B)=:Y$}. 
(As pointed out in the introduction, we do not use the notation~$\phi^*$ for the restriction of rings~$f_*$, to avoid confusion with extension of scalars~$f^*$.)
Since a proper affine morphism is necessarily finite, we see that $f$ is quasi-perfect if and only if $A$ admits a finite resolution by finitely generated projective $B$-modules.
Since $\Rder f_*  =  f_* \cong {}_BA \otimes_A - $, the right adjoint $f\uu{1}$ has the following reassuring description as a right derived functor: $f\uu{1}= \Rder \Hom_B({}_BA, -)$.
\end{Exa}

\begin{Exa}
	Of course, not all scheme maps $f:X\to Y$ are quasi-perfect. For instance, the affine morphism $f:=\Spec(\phi)$, with $\phi: \bbZ[t]/(t^2)\to \bbZ $ sending the variable $t$ to zero, is not; see \cite[Ex.\,6.5]{Neeman96}. Note that $f$ is finite, but indeed $\Rder f_*(\mathcal O_X)$  is not compact, as $\bbZ$ has infinite projective dimension over $\bbZ[t]/(t^2)$.
\end{Exa}

\begin{Exa}
\label{Exa:warning}%
Let $R$ be a Gorenstein local ring of Krull dimension~$d$ and~$k$ its residue field. Consider as above the morphism $f:\Spec(k)\to \Spec(R)$ and the induced functor $f^*:\Der(R)\to \Der(k)$. It does not satisfy Grothendieck-Neeman duality in general, that is, unless $f_*(k)$ is perfect, \ie $R$ is regular. However, $\DO=f\uu{1}(\unit)=\mathrm{RHom}_R(k,R)\simeq \Sigma^{-d}k$ because $R$ is Gorenstein, see~\cite[\S\,18]{Matsumura86}. In this case, $\DO$ is invertible despite failure of Grothendieck-Neeman duality.
\end{Exa}

\begin{Exa} \label{Exa:Brave_New}
The affine situation of Example~\ref{Exa:AG_affine} can be generalized in the Brave New direction (or in the Differential Graded direction), as in Example~\ref{exa:BraveNew}.  That is, we may consider $\phi: B\to A$ to be a morphism of commutative $\Sphere$-algebras (or commutative dg-rings).  One can check that $\phi$ induces a functor $f^*:=A\otimes_B-\colon \Der(B)\to \Der(A)$ satisfying our basic Hypotheses.  As before, its right adjoint $f_*$ is obtained simply by considering $A$-modules as $B$-modules through~$\phi$ and the next right adjoint $f\uu{1}$ is given by the formula $f\uu{1}= \Hom_B(A,-)$.  (All functors considered here are derived from appropriate Quillen adjunctions.) Since $\Der(A)^c$ is the thick subcategory generated by~$A$, we see by Neeman's criterion that, as before, $f^*$ satisfies Grothendieck-Neeman duality iff $f_*(A)$ is compact. 
\end{Exa}

\begin{Exa}
\label{Exa:poly-algebra}
Let $k$ be a field and consider the inclusion $\phi:k[x^n]\hook k[x]$ of graded $k$-algebras.  Since $k[x]$ is a free $k[x^n]$-module with homogeneous basis $\{1,x,\ldots,x^{n-1}\}$ we see that $f_*(\unit)$ is compact, and hence $f^*$ satisfies Grothendieck-Neeman duality.  Moreover, one easily checks that $\DO = \Hom_{k[x^n]}(k[x],k[x^n]) \cong \Sigma^{n-1} k[x] = \Sigma^{n-1}\unit$.  From a more abstract point of view, $f_*(\unit) \cong \bigoplus_{i=0}^{n-1} \Sigma^{-i} \unit$, and since $\Der(k[x])$ is generated by the unit, the relative dualizing object $\DO$ is characterized (cf.~\eqref{eq:i-DO}) by $f_*(\DO) \cong [f_*(\unit),\unit] = [\bigoplus_{i=0}^{n-1} \Sigma^{-i} \unit,\unit] = \bigoplus_{i=0}^{n-1} \Sigma^{i} \unit$.  Hence $\DO \cong \Sigma^{n-1}\unit$.  Note that if we regard $k[x^n] \to k[x]$ as a map of ungraded commutative rings and consider the extension-of-scalars functor $f^*$ between ordinary derived categories, then $\DO\cong\unit$.
\end{Exa}

\bigbreak
\section{The Wirthm\"uller isomorphism}
\label{se:Wirth}%

When we are in the Grothendieck-Neeman situation, \ie when we have five adjoints $f\dd{1}\adj f^*\adj f_*\adj f\uu{1}\adj f\dd{-1}$, the relative dualizing object~$\DO$ is remarkably ``close'' to being $\otimes$-invertible, a fact which perhaps deserves separate statement.

\begin{Prop} \label{Prop:rigid_implies_invertible}
Assume Grothendieck-Neeman duality (Thm.\,\ref{thm:GN}).
Then $f_*(\DO)$ is compact in~$\cat D$.
Moreover, $\DO$ is compact in~$\cat C$ if and only if it is $\otimes$-invertible.
\end{Prop}

\begin{proof}
We have $f_*(\DO)\cong f\dd{1}(\unit)$ by the ur-Wirthm\"uller isomorphism~\eqref{eq:ur-Wirth1}. Moreover, $\unit$ is assumed to be compact and $f\dd{1}$ preserves compact objects by Proposition~\ref{prop:cpt_coprod}, because its right adjoint~$f^*$ preserves coproducts by hypothesis.

Invertible objects are always rigid, hence compact under our assumptions.
Conversely, if $\DO$ is rigid there is an isomorphism $\dual(\DO) \otimes \DO \stackrel{\sim}{\to} \ihomcat{C}(\DO,\DO)$ and therefore by~\eqref{eq:ihom(w,w)1} an isomorphism $\dual(\DO)\otimes \DO\simeq \unit$, hence $\DO$ is invertible.
\end{proof}

We are now ready to prove Theorem~\ref{thm:Wirthmueller}, which abundantly characterizes the situations when~$\DO$ \emph{does} become invertible.

\begin{proof}[Proof of the Wirthm\"uller isomorphism Theorem~\ref{thm:Wirthmueller}]
The equivalent formulations of~(W\ref{it:far-left}) hold by Corollary~\ref{cor:adjoints} and similarly for~(W\ref{it:far-right}), together with Proposition~\ref{prop:cpt_coprod}. The equivalence
(W\ref{it:w-cpt})$\Leftrightarrow$(W\ref{it:w-inv}) holds by Proposition~\ref{Prop:rigid_implies_invertible}.
If $f\uu{1}$ preserves compacts then obviously $\DO=f\uu{1}(\unit)$ is compact, hence we have (W\ref{it:far-right})$\Rightarrow$(W\ref{it:w-cpt}).
Conversely, we can see from Grothendieck duality $f\uu{1}\cong \DO \otimes f^*$ that (W\ref{it:w-cpt})$\Rightarrow$(W\ref{it:far-right}), as our $f^*$ always preserves compacts.

Let us show (W\ref{it:far-left})$\Rightarrow$(W\ref{it:far-right}). Thus we now have six adjoints $f\uu{-1}\dashv f\dd{1}\dashv f^* \dashv f_*\dashv f\uu{1}\dashv f\dd{-1}$ and we want to show that $f\uu{1}$ preserves compacts.
Since $f\uu{-1}$, $f\dd{1}$, $f^*$ and $f_*$ have two-fold right adjoints they must preserve compacts by Proposition~\ref{prop:cpt_coprod} (their right adjoints preserve coproducts).
By restricting to compacts, we have four consecutive adjoints $f\uu{-1}|_{\cat D^c} \;\dashv\; f\dd{1} |_{\cat C^c} \;\dashv\; f^*|_{\cat D^c} \;\dashv\; f_*|_{\cat C^c}$. Now recall from Remark~\ref{Rem:selfduality_on_cpts} that $\dual=\ihom(-,\unit)$ defines a duality on compact objects, hence conjugating with it turns left adjoints into right adjoints. Furthermore, since $f^*\dual=\dual f^*$, the original functor $f^*|_{\cat D^c}$ is fixed by conjugation by~$\dual$. The above four adjoints therefore yield (several isomorphisms, like $\dual(f\dd{1}|_{\cat C^c})\dual\cong f_*|_{\cat C^c}$, and) the following five consecutive adjoints between $\cat D^c$ and~$\cat C^c$
\[
f\uu{-1}|_{\cat D^c}
\;\dashv\; f\dd{1} |_{\cat C^c}
\;\dashv\; f^*|_{\cat D^c}
\;\dashv\; f_*|_{\cat C^c}
\;\dashv\; \dual(f\uu{-1} |_{\cat D^c})\dual \,.
\]
The right-most functor is the unpredicted one. We can now apply Lemma~\ref{Lem:trick_with_adjoints}\,\eqref{it:right-trick} for $F:=f_*$ to show that its right adjoint $f\uu{1}$ preserves compacts, as desired.

Clearly (W\ref{it:w-inv})$\Rightarrow$(W\ref{it:Wirth}) because of the ur-Wirthm\"uller isomorphism~\eqref{eq:ur-Wirth}. Now assume~(W\ref{it:Wirth}) instead, \ie that there exists a $\otimes$-invertible $\omega\in\cat C$ such that $f_*(\omega\otimes-)$ is left adjoint to~$f^*$. Then the formulas~\eqref{eq:twists} make sense with $\omega$ instead of~$\DO$, yielding well-defined functors $f\uu{n}:\cat D\to \cat C$ and $f\dd{n} : \cat C\to \cat D$ for all $n\in \bbZ$:
\begin{align*}
 f\uu{n} := \omega^{\otimes n} \otimes f^*
\quad \quad
f\dd{n} := f_*(\omega^{\otimes n}\otimes -)
\quad
 \textrm{ with }
f\uu{0}:=f^*
\textrm{ and }
f\dd{0}:= f_* \,.
\end{align*}
Moreover, for all $n\in \bbZ$ we obtain the required adjunctions
$f\uu{n}\dashv f\dd{-n} \dashv f\uu{n+1}$
by variously composing the adjunction $f^*\dashv f_*$ with the appropriate power of $(\DO\otimes -) \dashv (\DO\inv \otimes -)$ or $(\DO\inv \otimes -) \dashv (\DO \otimes -)$.
Thus (W\ref{it:Wirth})$\Rightarrow$(W\ref{it:ad-lib}).

If~(W\ref{it:ad-lib}) holds then we have (W\ref{it:far-left})-(W\ref{it:far-right}) because then every functor in the tower must preserve products, coproducts and compacts.
The uniqueness of adjoints implies that whenever $f^*\dashv f_*$ sprouts a doubly infinite tower of adjoints this is necessarily given by the formulas~\eqref{eq:twists}, because in that case $\DO$ is invertible.

As the alert (or record-keeping) reader must have noticed, we have proved that the conditions (W\ref{it:far-left})-(W\ref{it:ad-lib}) are all equivalent, and that they imply~\eqref{eq:twists}. It remains to verify the claimed sufficient conditions, \ie the ``finally'' part.

For completeness, let us first recall (see Lemma~\ref{Lem:faithful_vs_cofinal} below) that $f_*$ is faithful if and only if $f^*$ is surjective on objects, up to direct summands.  Moreover, if this is the case then the counit $\eps_x:f^*f_*(x)\to x$ is a split epi for all $x\in \cat C$.  Therefore if $f_*(x)$ is compact then $x$ must be as well, because $f^*$ preserves compacts.  This shows the implication (1)$\Rightarrow$(2).  Clearly~(2) implies~(3), as they have the same conclusion but the hypothesis in~(2) is weaker.  To conclude the proof of the theorem, it suffices to show that~(3) implies~(W\ref{it:w-cpt}).  As we are in the context of Grothendieck-Neeman duality, we can use the ur-Wirthm\"uller equation
$ f_*(\DO\otimes x) \cong  f\dd{1}(x) $.
Since $f\dd{1}$ preserves compacts, it implies that $f_*(\DO\otimes x)$ is compact whenever $x$ is a compact object of~$\cat C$, so by~(3) we can conclude that $\DO$ is compact, that is~(W\ref{it:w-cpt}).
\end{proof}

\begin{Lem} \label{Lem:faithful_vs_cofinal}
Let $F:\cat S \adjto \cat T: G$ be adjoint exact functors between triangulated categories.
Then $G$ is faithful if and only if $F$ is surjective up to direct summands, that is, if and only if every object $x\in \cat T$ is a retract of $F(y)$ for some $y\in \cat S$. Moreover, this is equivalent to the counit of adjunction admitting a (possibly unnatural) section at each object.
\end{Lem}
\begin{proof}(Cf.\,\cite[Prop.\,2.10]{Balmer11}.)
For any $x\in \cat T$, let $FG(x)\stackrel{\eps_x\;\;}{\to} x \stackrel{\alpha}{\to} y \to \Sigma FG(x)$ be an exact triangle containing $\eps_x: FG(x)\to x$, the counit of adjunction at~$x$.  Then $G(\eps_x)$ is a split epi by one of the unit-counit relations.  Also, $\alpha\eps_x=0$ hence $G(\alpha)G(\eps_x)=0$. Together these facts imply $G(\alpha)=0$.  Now if $G$ is faithful we have $\alpha=0$ and therefore, by the exact triangle, $\eps_x$ is a split epi (cf.~\cite[Cor.\,1.2.7]{Neeman01}).  In particular, $x$ is a retract of $FG(x)$.  Hence $G$ faithful implies that $F$ is surjective up to direct summands.

Conversely, assume $F$ is surjective up to direct summands:  for every $x\in \cat T$ we can find an $x'\in \cat T$, a $y\in \cat S$, and an isomorphism $x\oplus x' \cong F(y)$.  By the other unit-counit relation, the morphism $\eps_{F(y)}: FGF(y) \to F(y)$ is a split epi.  By the naturality of $\eps$ and the additivity of the functors, the morphisms $\eps_{F(y)}$ and $\eps_{x}\oplus \eps_{x'}$ are isomorphic, hence $\eps_x$ must be an epi.
As $x\in \cat T$ was arbitrary, this proves $G$ faithful by \cite[Thm.\ IV.3.1]{MacLane98}.
\end{proof}

\begin{Rem}
\label{rem:FHM}%
The important article Fausk-Hu-May~\cite{FauskHuMay03} also deals with Gro\-then\-dieck duality and Wirthm\"uller isomorphisms, without assuming the categories to be triangulated until their final section. In some sense, we take over where they leave things and the picture becomes much simpler, as we now explain.

Fausk-Hu-May assume given two pairs of adjoints: the original $f^*\adj f_*$ and another one $f_{{!}}\adj f^{{!}}$. This is motivated by ``Verdier-Grothendieck duality" in algebraic geometry (which we do not consider here). They mainly study two special cases:
\begin{enumerate}[(1)]
\item
\label{it:FHM-G}%
The case $f_{{!}}=f_*$ or ``Grothendieck context", which reads $\qquad f^*\adj f_*\adj f^!$.
\item
\label{it:FHM-W}%
The case $f^{{!}}=f^*$ or ``Wirth\-m\"uller context", which reads $\,f_!\adj f^*\adj f_*$.
\end{enumerate}
Although they explicitly say that both cases can happen simultaneously, their paper rather stresses the separation between the two contexts (first in the notation, since $f_!\adj f^*\adj f_*\adj f^!$ would collide with $f_!\adj f^!$, but more systematically in the presentation: the two cases are deemed ``deceptively similar, but genuinely different"). However, for triangulated categories, our new condition~(GN\,\ref{it:new}) says that the mere existence of a left adjoint to~$f^*$ already forces Grothendieck duality, even before asking whether this left adjoint is a twisted form of~$f_*$. In other words, \eqref{it:FHM-W} is not genuinely different from~\eqref{it:FHM-G}, it actually \emph{implies}~\eqref{it:FHM-G}!

Also, the Wirthm\"uller context~\eqref{it:FHM-W} assumes the existence of some object $C$ with $f_*(\unit)\simeq f\dd{1}(C)$.
Then \cite[Thm.\,8.1]{FauskHuMay03} establishes a Wirthm\"uller isomorphism $f_*\simeq f\dd{1}(C\otimes-)$ comparable to our~(W\ref{it:Wirth}). However, in each example, such a ``Wirthm\"uller object" $C$ needs to be constructed by hand. For instance, in equivariant stable homotopy such a construction is done in a separate article~\cite{May03}, sequel to~\cite{FauskHuMay03}. Moreover, the relation $f_*(\unit)\simeq f\dd{1}(C)$ does not characterize~$C$ uniquely, a priori. Our approach avoids the mysterious object~$C$ altogether: The relative dualizing object~$\DO$ and the ur-Wirthm\"uller isomorphism~\eqref{eq:ur-Wirth} exist in the more general setting of Grothendieck duality, and when a Wirthm\"uller isomorphism exists, we can simply take $C$ to be the inverse of the relative dualizing object~$\DO$.

We did not find any example with non-rigid Wirthm\"uller object~$C$. Hence, the following result essentially subsumes the Wirthm\"uller context of~\cite{FauskHuMay03} into ours (again, for rigidly-compactly generated categories).
\end{Rem}

\begin{Prop}
Suppose that the basic adjunction $f^*:\cat D\adjto \cat C:f_*$ as in Hypothesis~\ref{hyp:base} fits in a ``Wirthm\"uller context" in the sense of~\cite{FauskHuMay03}, \ie suppose that $f^*$ has a left adjoint~$f\dd{1}$ (denoted~$f_!$ in~\cite{FauskHuMay03}) and that there exists an object $C\in\cat C$ such that $f\dd{1}(C)\simeq f_*(\unit)$. Then its dual $\ihomcat{C}(C,\unit)$ is isomorphic to~$\DO$.

If moreover $C$ is compact (\ie rigid), as is commonly the case in examples, then $\DO$ and $C$ are invertible and $C\cong \DO\inv$. In other words, a Wirth\"uller context with compact Wirthm\"uller object~$C$ only happens in the case of the infinite tower of adjoints (Theorem~\ref{thm:Wirthmueller}) and then $C$ must be the inverse of the canonical object~$\DO$.
\end{Prop}

\begin{proof}
By~\cite[Thm.\,8.1]{FauskHuMay03}, the Wirthm\"uller context yields an isomorphism $f\dd{1}(C \otimes -)\simeq f_*$.
Taking right adjoints (which exist by Theorem~\ref{thm:base}), we get
$$
\ihomcat{C}(C,f^*(-))\simeq f\uu{1}.
$$
Evaluating at~$\unit\in\cat D$, we obtain the desired $\ihomcat{C}(C,\unit)\simeq f\uu{1}(\unit)=\DO$. If moreover $C$ is rigid, then so is its dual~$\DO$. So, by Theorem~\ref{thm:Wirthmueller}, $\DO$ must be invertible.
\end{proof}

\begin{Exa}[Equivariant homotopy theory]
\label{Exa:eqhtpy}%
Let $H$ be a closed subgroup of a compact Lie group~$G$ and let $f^*:\SH(G)\to \SH(H)$ denote the restriction functor from the equivariant stable homotopy category of (genuine) $G$-spectra to that of $H$-spectra, as in Example~\ref{exa:SH}.
Then $f^*$ provides an example of Theorem~\ref{thm:Wirthmueller}.
The relative dualizing object $\DO$ is the $H$-sphere $S^L$ where $L$ denotes the tangent $H$-representation
of the smooth $G$-manifold $G/H$ at the identity coset $eH$ (see \cite[Chapter III]{LMS86}).
The ur-Wirthm\"uller isomorphism reads $G_+ \wedge_H X \simeq F_H(G_+,X \wedge S^L)$ and provides the
well-known Wirthm\"uller isomorphism between induction and coinduction, up to a twist by $S^L$.
If $H$ has finite index in $G$ (\eg\ if $G$ is a finite group) then $L=0$ and $\DO \cong \unit$.
\end{Exa}

\begin{Exa}[Finite group schemes]
	Let $H$ be a closed subgroup of a finite group scheme $G$ and consider their stable representation categories, as in Example~\ref{exa:Stab}.
	As discussed in \cite[Chapter~8]{Jantzen87},
	the restriction functor $f^*:\Stab(\kk G) \ra \Stab(\kk H)$ provides another example of Theorem~\ref{thm:Wirthmueller}.
	If $\delta_G$ denotes the unimodular character of the finite group scheme $G$ then
	the relative dualizing object $\DO$ is $\left.\delta_G\right|_H \cdot \delta_H^{-1}$.
	A finite group scheme is said to be ``unimodular'' if its unimodular character is trivial, which is equivalent to the group algebra being a symmetric algebra. This is the case for instance for (discrete) finite groups.
\end{Exa}

\begin{Exa}[Motivic homotopy theory] \label{Exa:mothtpy}
	Let $\kk$ be a field of characteristic zero, and let $\SHA(\kk)$ denote the stable $\bbA^1$-homotopy category over $\kk$,
	as in Example~\ref{exa:SHk}.
	For any finite extension $i:\kk\hookrightarrow \kk'$, the
	base change functor $i^*:\SHA(\kk) \to \SHA(\kk')$ provides another example of
	Theorem~\ref{thm:Wirthmueller}. In this example, the relative dualizing object
	$\DO$ is the unit object $\unit$. See \cite{Hu01} for further details.
\end{Exa}
\begin{Exa}[Cohomology rings of classifying spaces] \label{exa:coho-rings}
Let $H \hook G$ be an inclusion of connected compact Lie groups, and consider the restriction map $H^*(BG;\bbQ) \to H^*(BH;\bbQ)$ on rational cohomology rings -- regarded as commutative dg-algebras over $\bbQ$ with zero differentials.  We are then in the situation of Example~\ref{Exa:Brave_New}, with $f^*:\Der(H^*(BG;\bbQ)) \to \Der(H^*(BH;\bbQ))$ the derived extension-of-scalars.  Since the cohomology rings are polynomial algebras, it follows from \cite[\S 11]{DwyerWilkerson98} and Venkov's theorem that $H^*(BH;\bbQ)$ is a finitely generated free $H^*(BG;\bbQ)$-module, and thus a compact object of the derived category $\Der(H^*(BG;\bbQ))$.  Hence we have GN-duality.  In fact, $\omega_f = \Sigma^{-d}\unit$ where $d=\dim(G/H)$.  Indeed, the Eilenberg-Moore spectral sequence associated to the fibration $G/H \to BH \to BG$ collapses and provides an isomorphism $H^*(BH;\bbQ) \simeq H^*(BG;\bbQ) \otimes_\bbQ H^*(G/H;\bbQ)$ (cf.~\cite[Thm.\,8.1]{McCleary01}).  Moreover, since our groups are connected, the manifold $G/H$ is orientable.  Hence, by Poincar\'{e} duality its Betti numbers are symmetric and one sees that the graded dual of $H^*(G/H;\bbQ)$ is $\Sigma^{-d}H^*(G/H;\bbQ)$.  Extending scalars along $\bbQ \to H^*(BG;\bbQ)$, it follows that $\Sigma^{-d}H^*(BH;\bbQ)$ is the dual of $H^*(BH;\bbQ)$ in the derived category $\Der(H^*(BG;\bbQ))$ and the identification $\DO \cong \Sigma^{-d}\unit$ follows (cf.~\eqref{eq:i-DO}).

		As a concrete example, take $G=U(n)$ and let $T\le G$ be the maximal torus consisting of the diagonal unitary matrices.  The Weyl group $W=N_G(T)/T$ is the symmetric group $S_n$ and acts on $T$ by permuting the (diagonal) entries.  The cohomology ring $H^*(BT;\bbQ)$ is the polynomial algebra $\bbQ[x_1,\ldots,x_n]$ with all generators in degree 2 and with the Weyl group acting by permuting the generators.  Similarly, $H^*(BU(n);\bbQ)$ is a polynomial algebra $\bbQ[c_1,\ldots,c_n]$ with generators the universal Chern classes $c_i$ (of degree $2i$).  The map $H^*(BU(n);\bbQ) \to H^*(BT;\bbQ)$ is a monomorphism and provides an isomorphism $\bbQ[c_1,\ldots,c_n] \xra{\sim} \bbQ[x_1,\ldots,x_n]^{W}$	onto the subalgebra of $W$-invariant polynomials (\aka symmetric polynomials), sending $c_i$ to the $i$th elementary symmetric polynomial.  Here $d=\dim(G/T)=n^2-n$.  In this example, one can see directly that $\DO = \Sigma^{-n(n-1)}\unit$ by the same method as in Example~\ref{Exa:poly-algebra} using the fact that $\bbQ[x_1,\ldots,x_n]$ is a free $\bbQ[c_1,\ldots,c_n]$-module with basis given by the collection of monomials $\{ x_1^{a_1}x_2^{a_2}\cdots x_n^{a_n} \mid 0 \le a_i \le n-i \}$.  Just note that the degrees of the monomial generators range from $0$ to $n(n-1)$ (recall that $|x_i|=2$) and that the number of monomial generators of degree $i$ equals the number of generators of degree $n(n-1)-i$.
	\end{Exa}
\begin{Exa}[Highly-structured cochains]
		Fix a field $\kk$ and let $H\kk$ denote the associated Eilenberg-MacLane spectrum.  For any space $X$, the function spectrum $C^*(X) := F(\Sigma^\infty X_+,H\kk)$ is a highly structured commutative $H\kk$-algebra and we can consider its derived category $\Der(C^*(X))$ as in Example~\ref{exa:BraveNew}.  Note that $\pi_{-n}(C^*(X)) \cong H^n(X;\kk)$ and we think of $C^*(X)$ as the spectrum of ``cochains'' on~$X$.  In particular, for a compact Lie group $G$ we can consider the derived category $\Der(C^*(BG))$.  For any closed subgroup $H \le G$, the map $BH \to BG$ induces a map of commutative $H\kk$-algebras $C^*(BG) \to C^*(BH)$ and we are in the situation of Example~\ref{Exa:Brave_New}.  It follows from the fact that $G/H$ is a finite CW-complex, that $C^*(BH)$ is compact when regarded as a $C^*(BG)$-module.  Hence we have GN-duality.  Moreover, as discussed in \cite[\S6]{BensonGreenlees14}, the relative dualizing object is invertible and we have an isomorphism $\DO \cong F(BH^{-L},H\kk)$ where $L$ denotes the tangent $H$-representation of $G/H$ at the identity coset (already seen in Example~\ref{Exa:eqhtpy}) and $BH^{-L}$ is the associated Thom spectrum.  Moreover, if our groups are connected then this simplifies to $\DO \cong \Sigma^{-d}(C^*(BH))=\Sigma^{-d}\unit$ where $d=\dim(G/H)$.  

		In fact, when $\kk=\bbQ$ this example can be united with the previous Example~\ref{exa:coho-rings}.  Indeed, for any compact connected Lie group $G$, the results of~\cite{Shipley07}, together with the intrinsic formality of the polynomial algebra $H^*(BG;\bbQ)$, provides an equivalence $\Der(C^*(BG)) \cong \Der(H^*(BG;\bbQ))$ under which  the (derived) extensions of scalars along $C^*(BG)\to C^*(BH)$ and along $H^*(BG;\bbQ)\to H^*(BH;\bbQ)$ coincide.
\end{Exa}
\begin{Rem}
		In fact, the last two examples can be connected with Example~\ref{Exa:eqhtpy} by using Greenlees and Shipley's work on algebraic models for free rational $G$-spectra. See \cite{GreenleesShipley14,Greenlees14pp} for details.
\end{Rem}

Finally, we provide an example of a functor $f^*:\cat D\to \cat C$ satisfying Grothendieck-Neeman duality for which $\DO$ is \emph{not} invertible.

\begin{Exa}
Let $\phi: B\to A$ be a morphism of commutative rings, as in Example~\ref{Exa:AG_affine}, and assume that the induced pull-back functor $f^*= \Spec(\phi)^*= A\otimes^\Lder_B-: \Der(B)\to \Der(A)$ satisfies Grothendieck-Neeman duality, \ie that $A$ is (finitely generated and) perfect over~$B$ (see Ex.\,\ref{Exa:AG_affine}).
For simplicity assume that $B=\kk$ is a field, so that~$A$ is a finite-dimensional commutative $\kk$-algebra.
In this case $\DO $ is the \mbox{$A$-module} $\Rder \Hom_B(A,B)=\Hom_\kk(A,\kk)=:A^\star$, the $\kk$-linear dual of~$A$, seen as an object of~$\Der(A)$. As $A$ has Krull dimension zero, the Picard group of $A$ is trivial, so $\DO$ is invertible (\ie perfect!) only if $A^\star \cong A$ as $A$-modules.
For an explicit example where this is \emph{not} the case, we can take the three-dimensional $\kk$-algebra $A:= \kk[t,s]/(s^2,t^2,st)$, which has the basis $\{1,s,t\}$;
then $A^\star$ has the dual basis $\{1^\star,s^\star,t^\star\}$ and the $A$-action determined by
$s\cdot s^\star = 1^\star,\
t\cdot t^\star = 1^\star,\
s\cdot t^\star = s\cdot 1^\star = t\cdot s^\star = t\cdot 1^\star = 0$.
Since $(s\cdot A^\star)\cap (t\cdot A^\star)\ni 1^\star\neq 0$, this intersection is non-zero. On the other hand, $(s\cdot A)\cap (t\cdot A)=0$, hence $A^\star\not\simeq A$ as an $A$-module.
(From a traditional point of view: if $A$ is a finite-dimensional algebra, $A^\star$ has finite $A$-injective dimension and in fact any of its finite injective resolutions is a dualizing complex for~$A$, in the classical sense; see~\cite[Prop.\,X.9.1(b) and Ex.\,X.9.10(b)]{BourbakiAlgebre07}.)
\end{Exa}

\bigbreak
\section{Grothendieck duality on subcategories}
\label{se:Groth}%

In this section we consider subcategories $\cat C_0\subset \cat C$ admitting a dualizing object~$\kappa$ and study the behavior of such structures under our functors~$f^*$ and~$f_*$.

\begin{Def}
\label{def:dualizing_object}%
Let $\cat C_0\subset \cat C$ be a \emph{$\cat C^c$-submodule}, \ie a thick triangulated subcategory of our big category~$\cat C$ such that $c\otimes x\in \cat C_0$ for all $x\in \cat C_0$ and all compact~$c\in\cat C^c$. An object $\kappa\in \cat C_0$ is called a \emph{dualizing object for~$\cat C_0$} if the \emph{$\kappa$-twisted duality} $\Dk:=\ihomcat{C}(-,\kappa)$ defines an anti-equivalence on~$\cat C_0$:
\begin{equation}
\label{eq:hom(-,kappa)}%
\Dk=\ihomcat{C}(-,\kappa):\cat C_0\op\isotoo \cat C_0.
\end{equation}
In Section~\ref{se:Beyond}, we will consider the more general situation of an ``external" dualizing object $\kappa\in\cat C$ by dropping the assumption that $\kappa$ belongs to the subcategory $\cat C_0$ itself.
(Note that if $\unit $ belongs to $ \cat C_0$ then necessarily $\kappa \cong \Dk(\unit)\in \cat C_0$.)
\end{Def}

\begin{Rem}
\label{rem:dualizing}%
Because $\cat C(x, \Dk(y)) \cong \cat C(x \otimes y, \kappa) \cong \cat C (y \otimes x, \kappa) \cong \cat C(y, \Dk(x))\cong \cat C\op(\Dk(x),y)$, we see that $\Dk$ is adjoint to itself and we have a canonical natural morphism
\begin{equation}
\label{eq:can_kappa}
\varpi_{\kappa}: x \longrightarrow \Dk \Dk (x)
\end{equation}
for all $x\in \cat C$, which is both the unit and the counit of this self-adjunction. It satisfies $\Dk(\varpi)\circ\varpi_{\Dk}=\id_{\Dk}:\Dk\to \Dk\Dk\Dk\to \Dk$ by the unit-counit relation.
We say that $x\in\cat C$ is \emph{$\kappa$-reflexive} if this morphism $\varpi_\kappa$ is an isomorphism (at~$x$).
\end{Rem}

\begin{Lem} \label{Lem:duality_criterion}
For a $\cat C^c$-submodule $\cat C_0\subset \cat C$, an object $\kappa\in\cat C_0$ is a dualizing object if and only if the following two conditions are satisfied:
\begin{enumerate}[\rm(i)]
\item For every $x\in\cat C_0$, the $\kappa$-twisted dual $\dual_\kappa(x)=\ihomcat{C}(x,\kappa)$ belongs to~$\cat C_0$.
\item Every $x\in \cat C_0$ is $\kappa$-reflexive, \ie $\varpi_\kappa:x\isoto \Dk\Dk(x)$.
\end{enumerate}
\end{Lem}

\begin{proof}
In the adjunction $\Dk:\cat C_0\op \adjto \cat C_0:\Dk$ the unit and counit are isomorphisms if and only if $\Dk$ is an equivalence.
\end{proof}

\begin{Exa} \label{Exa:trivial_duality}
For the subcategory $\cat C_0=\cat C^c$ of all rigids, an object $\kappa\in \cat C_0$ is dualizing if and only if it is invertible (cf.\ \cite[Cor.\,5.9]{FauskHuMay03}). In particular $\cat C_0=\cat C^c$ always admits $\kappa=\unit$ as dualizing object.
\end{Exa}

\begin{Lem} \label{Lem:extraction_of_rigids}
There is a canonical natural isomorphism
$\dual_{\kappa}(x) \otimes \dual(c) \cong \dual_\kappa (x \otimes c)$
for all $x,\kappa\in \cat C$ and all rigid $c\in \cat C^c$, where $\dual=\dual_{\unit}$ is the usual dual.
\end{Lem}
\begin{proof}
This is standard; see~\cite[Thm.\ A.2.5.(d)]{HoveyPalmieriStrickland97}.
The map $\ihom(x,\kappa)\otimes \ihom(c,\unit) \ra \ihom(x\otimes c,\kappa)$
is adjoint to the map
$\ihom(x,\kappa)\otimes \ihom(c,\unit) \otimes x \otimes c
\cong \ihom(x,\kappa) \otimes x \otimes \ihom(c,\unit) \otimes c
\xra{\ev\otimes \ev} \kappa \otimes \unit \cong \kappa$.
\end{proof}

\begin{Rem}
It follows that if $\kappa$ is a dualizing object for~$\cat C_0$ then so is $\kappa\otimes u$ for every $\otimes$-invertible~$u$. In algebraic geometry, dualizing complexes are unique up to tensoring by an invertible object; see~\cite[Lem.\,3.9]{Neeman10}. For a general~$\cat C_0$ this seems to be over-optimistic, although we can prove the following variant, replacing an equivalence of the form $u\otimes-$ by one of the form $\ihomcat{C}(u,-)$.
\end{Rem}

\begin{Prop}
\label{prop:unique}%
Let $\cat C_0\subseteq \cat C$ be a $\cat C^c$-submodule containing $\cat C^c$ (that is, $\unit\in\cat C_0$). Let $\kappa$ and $\kappa'$ be two dualizing objects for~$\cat C_0$. Let $u:=\Dkk(\kappa)=\ihomcat{C}(\kappa,\kappa')$ and $v:=\Dk(\kappa')=\ihomcat{C}(\kappa',\kappa)$, both in~$\cat C_0$. Then $v \cong \dual u$ and $u\cong \dual v$ and the restrictions to~$\cat C_0$ of the functors $F_u:=\ihomcat{C}(u,-)$ and $F_v:=\ihomcat{C}(v,-)$ yield mutually inverse equivalences $F_u:\cat C_0\overset{\sim}\longleftrightarrow\cat C_0:F_v$ such that $\kappa\cong F_u(\kappa')$.
Moreover, we have $\dual_{\kappa'}\cong F_v\circ \dual_{\kappa}\cong \dual_{\kappa}\circ F_u$.
\end{Prop}

\begin{proof}
Let us write $[-,-]$ for $\ihomcat{C}(-,-)$ to save space. From the equivalence $[-,\kappa]:\cat C_0\op\isoto \cat C_0$, one deduces a natural isomorphism
\begin{equation}
\label{eq:[,kappa]}%
\beta:[x,y]\isoto \big[[y,\kappa],[x,\kappa]\big]
\end{equation}
for $x,y\in\cat C_0$. Indeed, the morphism $\beta$ is (double) adjoint to the (double) evaluation $[x,y]\otimes[y,\kappa]\otimes x\too \kappa$. When tested on $\Homcat{C}(c,-)$ for $c\in\cat C^c$, we obtain
$$
\xymatrix@C=1em{
\Homcat{C}(c,[x,y]) \ar@{}[r]|-{\cong} \ar[d]^-{\Homcat{C}(c,\beta)}
& \Homcat{C}(c\otimes x,y) \ar@{..>}[r]
& \Homcat{C}([y,\kappa],[c\otimes x,\kappa])
\\
\Homcat{C}(c,\big[[y,\kappa],[x,\kappa]\big]) \ar@{}[r]|-{\cong}
& \Homcat{C}(c\otimes [y,\kappa],[x,\kappa]) \ar@{}[r]|-{\cong}
& \Homcat{C}([y,\kappa],\dual c\otimes [x,\kappa]) \ar[u]^-{\cong}.
}
$$
The dotted arrow thus obtained is nothing but the map on morphisms sets induced by the contravariant functor~$[-,\kappa]:\cat C_0\op\to \cat C_0$. Since the latter is an equivalence, this dotted map is a bijection hence so is the left vertical map. Since this holds for every $c\in\cat C^c$ and since $\cat C$ is compactly generated, the morphism~$\beta$ is an isomorphism. As the same is true for $[-,\kappa']$ we have a natural isomorphism
\begin{equation}
\label{eq:double-kappa}%
\big[[y,\kappa],[x,\kappa]\big] \cong [x,y] \cong \big[[y,\kappa'],[x,\kappa']\big]
\end{equation}
for all $x,y\in \cat C_0$. Since $\unit\in\cat C_0$, we have $[\kappa,\kappa]\cong \dual_\kappa^2(\unit)\cong\unit$ and similarly $[\kappa',\kappa']\cong\unit$.
Thus, plugging $x=\kappa$ and $y=\kappa'$ (resp.\ $x=\kappa'$ and $y=\kappa$) in~\eqref{eq:double-kappa} we get the announced isomorphisms $u\cong \dual v$ and $v\cong \dual u$.
Plugging instead $x=\unit$ and $y=\kappa$ (resp.\ $y=\kappa'$) in~\eqref{eq:double-kappa}, we obtain
\begin{equation}
\label{eq:uv}%
\kappa\cong[u,\kappa'] \qquadtext{and}\kappa'\cong [v,\kappa].
\end{equation}
Now compute the composite equivalence $\xymatrix@C=3em{\cat C_0 \ar[r]_-{\simeq}^-{[-,\kappa]} & \cat C_0\op \ar[r]_-{\simeq}^-{[-,\kappa']} & \cat C_0.}$ For all $x\in\cat C_0$,
$$
\big[[x,\kappa],\kappa'\big]\overset{\eqref{eq:uv}}\cong\big[[x,\kappa],[v,\kappa]\big]\overset{\eqref{eq:[,kappa]}}\cong [v,x].
$$
This proves that $F_v:=[v,-]\cong \dual_{\kappa'}\dual_{\kappa}$ and $F_u:=[u,-]\cong \dual_{\kappa}\dual_{\kappa'}$ define equivalences on~$\cat C_0$, which satisfy the desired relations since $\Dk^2\cong\Id\cong \Dkk^2$.
\end{proof}

\begin{Rem}
One can deduce from Proposition~\ref{prop:unique} that $u=\ihomcat{C}(\kappa,\kappa')$ is $\otimes$-invertible, and that $\kappa\otimes u\cong \kappa'$, if $\cat C_0$ satisfies any of the following properties:
\begin{enumerate}[(i)]
\item
\label{it:uv1}%
$u\otimes \cat C_0\subset \cat C_0$ (for instance if $u$ belongs to $\cat C^c$);
\item
\label{it:uv2}%
$\cat C_0$ cogenerates~$\cat C$, \ie for $t\in \cat C$ if $\Homcat{C}(t,x)=0$ for all $x\in \cat C_0$ then $t=0$;
\item
\label{it:uv3}%
if an object $a\in \cat C$, not necessarily in~$\cat C_0$, admits a natural isomorphism $\ihomcat{C}(a,x)\cong x$ for all $x\in \cat C_0$, then $a\cong \unit$.
\end{enumerate}
This is left as an easy exercise for the interested reader.
Note in particular that condition~\eqref{it:uv1} holds if $\cat C_0$ satisfies $\cat C_0\otimes \cat C_0\subseteq \cat C_0$.
This assumption appears, for example, in~\cite{BoyarchenkoDrinfeld13}; however,
$\cat C_0\otimes \cat C_0\subseteq \cat C_0$ is not true
in algebraic geometry for $\cat C_0= \Db(\coh X)$.
Nevertheless, this important example can also be derived from Proposition~\ref{prop:unique}:
\end{Rem}

\begin{Cor}[{\cite[Lem.\,3.9]{Neeman10}}]
\label{cor:unique-AG}%
If $X$ is a noetherian scheme admitting two dualizing complexes $\kappa$ and~$\kappa'$, then there exists a $\otimes$-invertible $\ell \in\Dperf(X)$ (a shift of a line bundle on each connected component of~$X$) such that $\kappa'\cong \kappa\otimes \ell$.
\end{Cor}
\begin{proof}
According to the definition used in \cite{Neeman10}, which is slightly more general than the classical one, a dualizing complex for $X$ is an (internal) dualizing object for $\cat C_0=\Db(\coh X)$.  It suffices to prove that $\ell := u = \ihom(\kappa,\kappa')$ is $\otimes$-invertible, since Proposition~\ref{prop:unique} then implies that $\kappa\cong \ihom(u,\kappa') \cong \Delta u\otimes \kappa'= u^{-1}\otimes \kappa'$.  To this end, we appeal to the following criterion: Over a noetherian scheme~$X$, a complex $x\in \Db(\coh X)$ is $\otimes$-invertible iff both
\begin{enumerate}
\item  $x$ is ($\unit$-)\-reflexive,  $x\stackrel{\sim}{\to}\dual^2(x)$, and
\item $\unit$ is $x$-reflexive, $\unit \stackrel{\sim}{\to} \dual^2_x(\unit)= \ihom(x,x)$.
\end{enumerate}
\noindent
For the affine case, see \cite[Cor.\,5.7]{AvramovIyengarLipman10}.  The criterion globalizes because the canonical morphism $\varpi$ commutes with localization to an open subscheme \cite[$\S$1.3 and Rem.\,1.5.5]{AvramovIyengarLipman11}; cf.~\cite[Thm.\,4.1.2]{CalmesHornbostel09}.

By Proposition~\ref{prop:unique}, we have isomorphisms $u\cong \dual(v) \cong \dual^2(u)$ and $\unit \cong \ihom (u,u)$.  In order to conclude that $u$ is $\otimes$-invertible by the above criterion, we must prove that the latter isomorphisms are instances of the canonical maps~$\varpi$.  This verification is easy for the second map, but for the first one  it appears to be rather involved.  Fortunately, we can avoid it altogether: By the local nature of~$\varpi$, we may reduce to the affine case where, by \cite[Prop.\,2.3]{AvramovIyengarLipman10}, the existence of \emph{any} isomorphism $x\cong \dual^2_y(x)$ implies that $x$ is $y$-reflexive (for $x,y\in \Db(\coh X)$).
\end{proof}

\begin{center}
*\ *\ *
\end{center}

Let us now ``move" the above subcategories with duality under~$f^*$ and~$f_*$. In order to do this, and for later use in Theorem~\ref{Thm:abstract_DGI_duality}, we need to clarify the following:
\begin{Thm}
\label{thm:duality_preserving}%
Let $f^*: \cat D\to \cat C$ be as in our basic Hypothesis~\ref{hyp:base} and $\kappa\in \cat D$. Recall the two adjunctions $f^* \dashv f_* \dashv f\uu{1} $ of Corollary~\ref{cor:base}, as well as their internal realizations~\eqref{eq:inthom-f^*-f_*} and~\eqref{eq:inthom-f_*-f^1}. The latter yields a canonical natural isomorphism
\begin{align} \label{eq:f_*-intertwines-dualities}
\zeta:\dual_\kappa\circ f_*  \isoto f_* \circ \dual_{f\uu{1}(\kappa)}.
\end{align}
This isomorphism is compatible with the canonical maps $\varpi$ of $\dual_{\kappa}$ and $\dual_{\kappa'}$ for $\kappa'=f\uu{1}(\kappa)$. This means that the following diagram commutes, for all $x\in\cat C$\,:
\begin{align} \label{eq:duality_preserving}
\begin{gathered}
\xymatrix{
f_*(x) \ar[d]_{\varpi_{f_*(x)}} \ar[r]^-{f_*(\varpi_x)} &
 f_* \dual_{\kappa'} \dual_{\kappa'} (x) \ar[d]_{\cong}^{\zeta_{\dual_{\kappa'}(x)}} \\
\dual_{\kappa} \dual_{\kappa} f_* (x) \ar[r]^\cong_-{\dual_{\kappa} (\zeta)} &
  \dual_{\kappa} f_* \dual_{\kappa'} (x).
}
\end{gathered}
\end{align}
In other words, $f_*:\cat C\to \cat D$ is a duality-preserving functor in the sense of~\cite{CalmesHornbostel09}.
\end{Thm}

\begin{proof}
It is far from obvious to verify that~\eqref{eq:duality_preserving} commutes, but fortunately this has already been proved, in even greater generality, in \cite[Thm.~4.2.9]{CalmesHornbostel09}.  Indeed, we are proving that the functor $f_*: \cat C\to \cat D$, or rather the pair $(f_*, \zeta)$, is a ``duality-preserving functor'' in the sense of \cite[Def.\,2.2.1]{CalmesHornbostel09} between the ``categories with (weak) duality'' $(\cat C,\dual_{\kappa'})$ and $(\cat D, \dual_\kappa)$.  The hypotheses $(\mathbf A_f)$, $(\mathbf B_f)$ and $(\mathbf C_f)$ of the cited theorem are all satisfied in our situation by virtue of Corollary~\ref{cor:base} (note that our $f\uu{1}$ is denoted $f^!$ in \emph{loc.\,cit.}).  To see that the conclusion of the cited theorem applies here, we must still verify that the natural maps we denote by~$\pi$ and $\zeta$ coincide with the homonymous maps of \cite{CalmesHornbostel09}.  For~$\pi$, it suffices to inspect the definitions in Proposition~\ref{Prop:right_proj_formula} and \cite[Prop.\,4.2.5]{CalmesHornbostel09} and note that they agree.  For~$\zeta$, we must compare our definition of~\eqref{eq:inthom-f_*-f^1} as a conjugate of~$\pi$ (which then specializes to~\eqref{eq:f_*-intertwines-dualities}) with the definition of $\zeta$ given in \cite[Thm.\,4.2.9]{CalmesHornbostel09}.  In more detail, we must show that our map~\eqref{eq:inthom-f_*-f^1} coincides with the composite
\[f_*\ihom(x,f\uu{1}y) \to \ihom(f_*x,f_*f\uu{1}y) \xra{\ihom(1,\eps)} \ihom(f_*x,y)\]
where the first map is the canonical map
$f_*\ihom(a,b) \to \ihom(f_*a,f_*b)$
induced by the (lax) monoidal structure on $f_*$.  This is readily checked from the definition of~\eqref{eq:inthom-f_*-f^1} in terms of $\pi$, together with the first diagram in~\eqref{eq:lax-pi-auxiliary}.
\end{proof}

\begin{Def}
\label{def:cpt_pullback}
If $\cat D_0$ is a $\cat D^c$-submodule of $\cat D$, we define its \emph{compact pull-back along $f_*$} as the following full subcategory of~$\cat C$:
\[
f^\#(\cat D_0):= \{ x\in \cat C\mid f_*(c \otimes x) \in \cat D_0 \textrm{ for all }c\in \cat C^c \} \,.
\]
One sees immediately that $f^\#(\cat D_0)$ is a $\cat C^c$-submodule of~$\cat C$, because $\cat C^c$ is one.
\end{Def}

We can use this compact pull-back $f^\#$ to rephrase Grothendieck-Neeman duality:
\begin{Prop} \label{Prop:Crelcpt_thick}
Let $f^*: \cat D\to \cat C$ be as in Hypothesis~\ref{hyp:base}.
\begin{enumerate}[\indent\rm(a)]
\item
The functor $f^*$ satisfies Grothendieck-Neeman duality (Theorem~\ref{thm:GN}) if and only if the following inclusion holds: $\cat C^c \subset f^\#(\cat D^c)$.
\item
If $\cat C^c = f^\# (\cat D^c)$ then we have the Wirthm\"uller isomorphism (Theorem~\ref{thm:Wirthmueller}).
\end{enumerate}
\end{Prop}

\begin{proof}
Hypothesis~(GN\,\ref{it:Neeman}) of Theorem~\ref{thm:GN-intro} reads: $f_*(c)\in\cat D^c$ for all $c\in\cat C^c$. Since $\unit\in\cat C^c$ and since $\cat C^c$ is stable under the tensor product, this condition is equivalent to: $f_*(c\otimes x)\in\cat D^c$ for all $x,c\in\cat C^c$. The latter exactly means $\cat C^c\subset f^\#(\cat D^c)$ by Definition~\ref{def:cpt_pullback}. Hence~(a).
For~(b), it suffices to note that $\cat C^c \supset f^\# (\cat D^c)$ is precisely the sufficient condition~\eqref{it:rel-cpt} in Theorem~\ref{thm:Wirthmueller}.
\end{proof}

Furthermore, compact pullback is compatible with composition of functors:
\begin{Prop}
\label{prop:trans_pullback}%
Consider two composable functors $\cat E \stackrel{g^*}{\longrightarrow} \cat D\stackrel{f^*}{\longrightarrow} \cat C$, both satisfying Hypothesis~\ref{hyp:base} and with composite $f^*g^*=:(gf)^*$, and let $\cat E_0$ be any $\cat E^c$-submodule of~$\cat E$.
Then $(gf)^\#(\cat E_0)=f^\# (g^\#(\cat E_0))$.
\end{Prop}

\begin{proof}
Notice that the composite $(gf)^*$ also satisfies Hypothesis~\ref{hyp:base}, and that its right adjoint must be $(gf)_*\cong g_* f_*$ by the uniqueness of adjoints.
Thus $x\in \cat C$ belongs to $(gf)^\#(\cat E_0)$ iff $g_*f_*(c\otimes x)\in \cat E_0$ for all $c\in \cat C^c$.

On the other hand: $x\in f^\#(g^\#(\cat E_0))$ iff $f_*(c\otimes x)\in g^\#(\cat E_0)$ for all $c\in \cat C^c$,
iff $g_*(d\otimes f_*(c\otimes x))\in \cat E_0$ for all $c\in \cat C^c$ and $d\in \cat D^c$.
By the projection formula~\eqref{eq:right_proj_formula}, we see that
$
g_*(d\otimes f_*(c\otimes x)) \cong g_*f_*(f^*(d) \otimes c \otimes x)
$.

Since each $f^*(d) \otimes c$ as above is compact in $\cat C$, and since every compact of $\cat C$ has this form (choose $d=\unit_\cat D$),  the two conditions on~$x$ are equivalent.
\end{proof}

Let us give an example of this compact pullback $f^\#$ in algebraic geometry.

\begin{Thm} \label{thm:cpt-pb-AG}
Let $f:X\to Y$ be a morphism of noetherian schemes and let $f^*:\cat D=\Dqc(Y)\too \Dqc(X)=\cat C$ be the induced functor (see Ex.\,\ref{exa:alg_geom_general}).
\begin{enumerate}[\indent\rm(a)]
\item
\label{it:f_*(coh)}%
Suppose that $f:X\to Y$ is proper. Then $f_*:\cat C\to \cat D$ maps $\Db(\coh X)$ into~$\Db(\coh Y)$.
Moreover, for every object $x\in\Db(\coh X)$ and every perfect $c\in\Dqc(X)^c$ we have $f_*(c\otimes x)\in \Db(\coh Y)$.
\item
\label{it:proj}%
Suppose that $f:X\to Y$ is projective. Then the following converse to~\eqref{it:f_*(coh)} holds\,: If $x\in \Dqc(X)$ is such that $f_*(c\otimes x)\in\Db(\coh Y)$ for every perfect $c\in\Dqc(X)^c$ then $x\in\Db(\coh X)$.
\end{enumerate}
In the notation of Definition~\ref{def:cpt_pullback}, we have for $f:X\to Y$ projective that
$$
f^\#\big(\Db(\coh Y)\big)=\Db(\coh X).
$$
\end{Thm}

\begin{proof}
For~\eqref{it:f_*(coh)}, the question being local in the base~$Y$, we can assume that $Y=\Spec(A)$ is affine. For any coherent sheaf~$\mathcal{F}\in\coh(X)$,  the $A$-modules $\RR^i f_*\mathcal{F}$ are finitely generated and vanish for $i>\!\!>0$, by~\cite[III.3.2.3]{EGA3}. It follows that $f_*\mathcal{F}=\RR f_*\mathcal{F}$ is bounded coherent. Hence so is $f_*(x)$ for any $x\in \Dqc(X)$ contained in the thick subcategory generated by coherent sheaves, which is precisely~$\Db(\coh X)$. The ``moreover part" follows immediately since $\Db(\coh X)$ is a $\Dperf(X)$-submodule of~$\Dqc(X)$, where $\Dperf(X)=\Dqc(X)^c$.

For~\eqref{it:proj}, in view of Proposition~\ref{prop:trans_pullback}, and since we can decompose $f$ into a closed immersion followed by the structure morphism $\bbP^n_Y\to Y$, we treat the two cases separately. For $f:X\to Y$ a closed immersion, the result is straightforward since $f_*$ itself detects boundedness and coherence. (One can reduce to $Y$ affine -- in any case, $f_*$ commutes with taking homology and detects coherence, so $f_*(x)\in\Db(\coh Y)$ forces $x$ to be bounded with coherent homology, \ie to be in~$\Db(\coh X)$.) For the projection~$f: \bbP^n_Y\to Y$, we can use the resolution of the diagonal $\cO_{\mathbb{\Delta}}$ in $\bbP^n\times_Y\bbP^n$ by objects of the form $p_1^*d_i\otimes p_2^* c_i$ where $p_1,p_2:\bbP^n\times \bbP^n\to \bbP^n$ are the two projections and $d_i=\Omega^i(i)$ and $c_i=\cO(-i)$ are vector bundles over~$\bbP^n$ and in particular compact objects; see~\cite{Beilinson83}. Hence, since every $x\in \Dqc(\bbP^n_Y)$
is isomorphic to $(p_2)_*(\cO_{\mathbb{\Delta}}\otimes p_1^*(x))$, we see that $x$ belongs to the thick subcategory generated by the $(p_2)_*\big(p_1^*d_i\otimes p_2^* c_i\otimes p_1^*(x))$. Computing the latter using the projection formula and flat base change (in the cartesian square for $\bbP^n\times_Y \bbP^n$ over~$Y$), we get
$$
(p_2)_*\big(p_1^*d_i\otimes p_2^* c_i\otimes p_1^*(x))
\;\;\cong\;\; (p_2)_*(p_1^*(x\otimes d_i))\otimes c_i
\;\;\cong\;\; f^*(f_*(x\otimes d_i))\otimes c_i\,.
$$
In particular, as soon as $f_*(x\otimes d_i)\in \Db(\coh Y)$ then all the objects above belong to $\Db(\coh\bbP^n_Y)$, hence so does~$x$ itself.
\end{proof}

\begin{Cor}
\label{cor:cpt-pb-AG}%
Let $f:X\to S$ be a projective morphism of noetherian schemes, with $S$ regular (for instance $S=\Spec(\kk)$ for $\kk$ a field). Then $\Db(\coh X)$ is equal to~$\SET{x\in\Dqc(X)}{f_*(c\otimes x)\in\Dperf(S)\textrm{ for all }c\in\Dperf(X)}$.
\end{Cor}

\begin{proof}
In this case, $\Dperf(S)=\Db(\coh S)$ and we can apply Theorem~\ref{thm:cpt-pb-AG}.
\end{proof}

\begin{center}
*\ *\ *
\end{center}

We now turn to the interaction between the two notions discussed above, namely that of dualizing object and that of compact pullback of subcategories.

\begin{Thm}\label{Thm:duality_iso}
Let $f^*\colon \cat D\to \cat C$ be as in Hypothesis~\ref{hyp:base}.
Let $\cat D_0\subset \cat D$ be a $\cat D^c$-submodule equipped with a dualizing object $\kappa\in \cat D_0$ (Def.\,\ref{def:dualizing_object}) and consider the following two possible properties of an object $x\in \cat C$:
\begin{enumerate}[\indent\rm(i)]
\item $x\in f^\#(\cat D_0)$.
\label{it:relative_compact}%
\item $x$ is $f\uu{1}(\kappa)$-reflexive: $x \stackrel{\sim}{\to} \dual_{f\uu{1}(\kappa)} \dual_{f\uu{1}(\kappa)} (x)$.
\label{it:reflexive}%
\end{enumerate}
Then we have\,:
\begin{enumerate}[\indent\rm(a)]
\item
\label{it:i=>ii}%
If $\unit\in\cat D_0$ then~\eqref{it:relative_compact} implies~\eqref{it:reflexive}.
\item
\label{it:ii=>i}%
If $\cat D_0$ consists precisely of the $\kappa$-reflexive objects of~$\cat D$, then~\eqref{it:reflexive} implies~\eqref{it:relative_compact}.
\end{enumerate}
\end{Thm}

\begin{proof}
Let us prove~\eqref{it:relative_compact}$\Rightarrow$\eqref{it:reflexive} when $\unit\in\cat D_0$, and write $\kappa':=f\uu{1}(\kappa)$ for short.
Consider $x,y ,c \in \cat C$ with $c$ compact, $y$ arbitrary, and $x\in f^\#(\cat D_0)$, which implies in particular that $f_*(x\otimes \dual c)\in \cat D^c$.
We obtain the following isomorphism:
\begin{align*}
\cat C (c, x)
&\;\;\cong\;\; \cat C (\unit , x \otimes \dual c) && c \in \cat C^c \textrm{ is rigid}
\\
&\;\;\cong\;\; \cat D \big(\unit , f_*(x \otimes \dual c) \big) && \unit\cong f^*\unit\textrm{ and } f^* \dashv f_* \\
&\;\;\cong\;\; \cat D \big(\Dk f_*(x\otimes \dual c) , \kappa \big) && \unit\textrm{ and } f_*(x\otimes \dual c)\in \cat D_0 \textrm{ and }\Dk:\cat D_0\op\isoto\cat D_0 \\
&\;\;\cong\;\; \cat D \big( f_* \Dkk(x\otimes \dual c), \kappa \big) &&
  \textrm{by~\eqref{eq:f_*-intertwines-dualities}, special case of~}\eqref{eq:inthom-f_*-f^1}
\\
&\;\;\cong\;\; \cat C \big( \Dkk(x\otimes \dual c), \kappa' \big) &&  f_* \dashv f\uu{1} \textrm{ and }f\uu{1}(\kappa)=\kappa'
\\
&\;\;\cong\;\; \cat C \big( \Dkk(x) \otimes c, \kappa' \big) &&
   c\in\cat C^c\textrm{ is rigid and Lemma~\ref{Lem:extraction_of_rigids}}  \\
&\;\;\cong \;\;  \cat C (c, \Dkk \Dkk (x)) && \otimes \dashv \ihomcat{C} \,.
\end{align*}
By following through this composite isomorphism, one can check that it is induced by the canonical map~\eqref{eq:can_kappa}.
Indeed, choosing an arbitrary morphism $\varphi:c\to x$ and writing $[-,-]$ for $\ihom(-,-)$, the relevant diagram can be checked using the naturality of the isomorphism
$f_*[a,f\uu{1}b] \cong [f_* a,b]$ in~\eqref{eq:inthom-f_*-f^1} with respect to the morphism $\unit\oto{\eta} c\otimes \dual c\,\otoo{\varphi\otimes 1}\, x\otimes \dual c$, together with the following three diagrams:
\[\xymatrix @R=3mm @C=7.5mm{ [c,\kappa']\otimes c \ar[dd]^{\ev} \ar[r]^{\sim}_-{\textrm{\ref{Lem:extraction_of_rigids}}}
    & [c \otimes \dual c, \kappa'] \ar[dd]^{[\eta,1]}
	& [\unit,\kappa'] \ar[dd]^{\sim} \ar[r]_-{f_*\adj f^{(1)}}
    & f\uu{1}f_*[\unit,\kappa'] \ar[r]^{\sim}_-{\eqref{eq:inthom-f_*-f^1}} & f\uu{1}[f_*\unit,\kappa] \ar[dd]_-{f^*\adj f_*} \\\\
		\kappa' \ar[r]^{\sim} & [\unit,\kappa']
		& \kappa' \ar@{=}[r] & f\uu{1}\kappa \ar[r]^{\sim} & f\uu{1}[\unit,\kappa]}
\]
\[\xymatrix @R=3mm @C=8.7mm{
 c \ar[rr]^-{\varphi} \ar[dd]^{\coev} && x \ar[rr]^-{\coev} && [[x,\kappa'],[x,\kappa']\otimes x] \ar[dd]^{[1,\ev]}\\\\
		[[x,\kappa'],[x,\kappa']\otimes c] \ar[rr]^-{[1,[\varphi,1]\otimes 1]} && [[x,\kappa'],[c,\kappa']\otimes c] \ar[rr]^-{[1,\ev]} && [[x,\kappa'],\kappa']].}\kern.6mm
\]
The first diagram can be checked using the definition of $[c,\kappa']\otimes c \cong [c\otimes \dual c,\kappa']$ (cf.~Lemma~\ref{Lem:extraction_of_rigids}) and dinaturality of coevaluation with respect to $\eta:\unit \ra c \otimes \dual c$.
Similarly, the second diagram can be checked using the definition of $f_*[\unit,\kappa'] \cong [f_* \unit,\kappa]$ together with dinaturality of coevaluation with respect to $\unit \ra f_*\unit$.
Finally, the last diagram follows from dinaturality of (co)evaluation and naturality of evaluation
with respect to~$\varphi:c\ra x$.
As $\cat C$ is compactly generated, the canonical map~\eqref{eq:can_kappa} is invertible for any $x$, as claimed.

For~\eqref{it:ii=>i}, in order to prove the conditional implication~\eqref{it:reflexive}$\Rightarrow$\eqref{it:relative_compact}, assume that $x\in \cat C$ is $\kappa'$-reflexive and let $c\in \cat C^c$ be a compact object.
We must show that $f_*(x\otimes c)$ belongs to~$\cat D_0$.
By the extra hypothesis, it suffices to show that $f_*(x \otimes c)\in \cat D$ is $\kappa$-reflexive.
We have the following composite isomorphism starting with $\kappa'$-reflexivity of~$x$ and $\unit$-reflexivity of~$c$\,:
\begin{align*}
f_* ( x  \otimes c )
&\;\; \cong \;\;  f_* \big( \dual_{\kappa'}  \dual_{\kappa'} (x ) \otimes \dual \dual(c) \big) && \\
&\;\; \cong \;\;  f_* \dual_{\kappa'} \dual_{\kappa'} (x \otimes c) && \textrm{Lemma~\ref{Lem:extraction_of_rigids}}\textrm{ twice}\\
&\;\; \cong \;\;  \dual_\kappa \dual_\kappa f_*(x \otimes c) && \eqref{eq:f_*-intertwines-dualities}\textrm{ twice.}
\end{align*}
To check that this isomorphism $f_*(x\otimes c) \isoto \dual_\kappa \dual_\kappa f_*(x\otimes c)$ coincides
with the canonical map~\eqref{eq:can_kappa}, it suffices to check that it is adjoint to the
evaluation map $[f_*(x\otimes c),\kappa]\otimes f_*(x\otimes c)\otoo{\ev} \kappa$. This can be accomplished
from the definitions by using dinaturality of evaluation with respect
to $f_*\dual_{\kappa'}(x\otimes c) \cong \dual_\kappa (f_*(x\otimes c))$ together with the following two commutative diagrams:
\[\xymatrix{
		f_*[a,f\uu{1}(b)]\otimes f_*(a) \ar[r]^-{\mathrm{lax}} \ar[d]^{\cong} & f_*([a,f\uu{1}(b)]\otimes a] \ar[r]^-{f_*\ev} & f_*f\uu{1}(b) \ar[d]^{\eps} \\
		[f_*(a),b]\otimes f_*(a) \ar[rr]^-{\ev}&& b
	}\]
and
$$
\xymatrix@R=4mm@C=5em{
  \dual_{\kappa}(x\otimes c)\otimes x\otimes c  \ar[r]^-{1\otimes\varpi_{\kappa}\otimes\varpi} \ar[dd]_-{\cong}
& \dual_{\kappa}(x\otimes c)\otimes \dual_{\kappa}^2x \otimes \dual^2c \ar[d]^-{\ \switch}
\\
& \dual_{\kappa}^2x\otimes \dual^2c\otimes \dual_{\kappa}(x\otimes c) \ar[d]_-{\cong}
\\
  \kappa
& \dual_{\kappa}^2(x\otimes c)\otimes \dual_{\kappa}(x\otimes c). \ar[l]_-{\ev}
}
$$
The first diagram can be checked in a straightforward manner by using the definition of $f_*[a,f^!b]\cong [f_*a,b]$ in~\eqref{eq:inthom-f_*-f^1}.
The second diagram can be checked using the definition of the maps $\varpi_{\kappa}:x\to \dual_{\kappa}^2x$ and $\varpi:c\to \dual^2c$ together with the fact that
$$
\xymatrix{
  [a,b]\otimes [a',b']\otimes a\otimes a' \ar[r]^-{\switch} \ar[d]_-{\cong}
& [a,b]\otimes a\otimes [a',b']\otimes a'  \ar[d]^-{\ev\otimes\ev}
\\
  [a\otimes a',b\otimes b']\otimes a\otimes a' \ar[r]^-{\ev}
& b\otimes b'
}
$$
commutes (which can be checked from the definition).

We conclude that indeed $f_*(x \otimes c)$ is $\kappa$-reflexive and therefore belongs to~$\cat D_0$.
\end{proof}

Let us give an example illustrating part~\eqref{it:ii=>i} of the theorem.
\begin{Cor}
Let $f:X\to \Spec(k)$ be a projective scheme over a field, and let $\DO\in \Dqc(X)$ denote the relative dualizing object for~$f^*$.
Then $x\in\Dqc(X)$ is $\DO$-reflexive iff
$\dim_k H^i f_*(c\otimes x) < \infty$ for all $c\in\Dperf(X)$ and~$i\in\bbZ$.

\end{Cor}
\begin{proof}
An object of $\Der(k)$ is $k$-reflexive iff its homology groups are all finite dimensional.
Now apply Theorem~\ref{Thm:duality_iso} \eqref{it:i=>ii} \&~\eqref{it:ii=>i} to $f^*: \Der(k)\to \Dqc(X)$.
\end{proof}

The next theorem is the main result of this section.

\begin{Thm}[Grothendieck duality]
\label{thm:rel-Groth}%
Let $f^*\colon \cat D\to \cat C$ be as in our basic Hypothesis~\ref{hyp:base} and let $\kappa\in\cat D$. Recall $f^*\adj f_*\adj f\uu{1}$ from Corollary~\ref{cor:base}.
Suppose that $f^*$ satisfies Grothendieck-Neeman duality (Theorem~\ref{thm:GN-intro}) and that $\cat D_0\subset \cat D$ is a $\cat D^c$-submodule which admits $\kappa\in\cat D_0$ as dualizing object (Definition~\ref{def:dualizing_object}). Then
\[ \kappa':=f\uu{1}(\kappa)\cong\DO\otimes f^*(\kappa) \]
 is a dualizing object for the $\cat C^c$-submodule $f^\#(\cat D_0)\subset \cat C$ (Definition~\ref{def:cpt_pullback}).
In particular, $f_*:f^\#(\cat D_0)\too \cat D_0$ is a duality-preserving exact functor between categories with duality, where $f^\#(\cat D_0)$ is equipped with the duality~$\Dkk$ and $\cat D_0$ with~$\Dk$.
\end{Thm}

\begin{proof}
We have $\kappa'\cong \DO\otimes f^*(\kappa)$ by formula~\eqref{eq:Groth}. Moreover, by Theorem~\ref{Thm:duality_iso}, every object of $f^\#(\cat D_0)$ is $\kappa'$-reflexive, hence by Lemma~\ref{Lem:duality_criterion} it remains to prove that $\kappa' = \DO\otimes f^*\kappa$ belongs to~$f^\#(\cat D_0)$ and that $\dual_{\kappa'}$ preserves~$f^\#(\cat D_0)$.
Indeed for every compact~$c\in\cat C^c$, we can use the projection formula~\eqref{eq:right_proj_formula} and the ur-Wirthm\"uller formula~\eqref{eq:ur-Wirth} to compute
\[
f_*(\DO \otimes f^*\kappa \otimes c)
\;\;\cong\;\; f_*(\DO \otimes c) \otimes \kappa
\;\;\cong\;\; f\dd{1}(c)\otimes \kappa \,;
\]
this object belongs to~$\cat D_0$ since $\kappa$ does, since $f\dd{1}$ preserves compacts and since~$\cat D_0$ is a $\cat C^c$-submodule. Hence $\kappa'=\DO \otimes f^*\kappa$ belongs to~$f^\#(\cat D_0)$.
Finally, let us show that $\Dkk$ preserves~$f^\#(\cat D_0)$.
For $x\in f^\#(\cat D_0)$ and $c\in\cat C^c$ we compute, using first Lemma~\ref{Lem:extraction_of_rigids} and the rigidity of~$c\in\cat C^c$ and then~\eqref{eq:f_*-intertwines-dualities}:
\begin{align*}
f_*(c\otimes \dual_{\kappa'} (x))
\;\;\cong\;\; f_*(\dual_{\kappa'}(x \otimes \dual c))
\;\;\cong\;\; \dual_\kappa (f_* (x \otimes \dual c)) \,.
\end{align*}
The latter belongs to~$\cat D_0$ since $f_*(x\otimes \dual c)$ does by definition of~$x\in f^\#(\cat D_0)$ and since $\Dk$ preserves~$\cat D_0$ by hypothesis. This shows $\dual_{\kappa'}(x)\in f^\#(\cat D_0)$ as wanted.
\end{proof}

\bigbreak
\section{Categories over a base and relative compactness}
\label{se:rel-cpt}%

We now want to analyse a relative setting.

\begin{Def}
\label{def:B-cat}%
Let $\cat B$ be a rigidly-compactly generated tensor triangulated category that we call the ``base". We say that $\cat C$ is a \emph{$\cat B$-category} if it comes equipped with a \emph{structure} functor $p^*:\cat B\to \cat C$ satisfying Hypothesis~\ref{hyp:base}.
\end{Def}

\begin{Def}
\label{def:B-mor}%
A \emph{morphism of $\cat B$-categories} $f^*:(\cat D,q^*)\to (\cat C,p^*)$ is
a functor $f^*:\cat D\to \cat C$ satisfying Hypothesis~\ref{hyp:base} together with an isomorphism $f^*q^*\cong p^*$.
By the uniqueness property of adjoint functors, this canonically spawns isomorphisms $q_* f_*\cong p_*$ and $f\uu{1}q\uu{1}\cong p\uu{1}$.
In particular, we have an isomorphism in~$\cat C$:
\begin{equation}
\label{eq:f-omegas}%
f\uu{1}(\omega_q) = f\uu{1} q\uu{1}(\unit_\cat B) \cong p\uu{1}(\unit_\cat B) = \omega_p\,.
\end{equation}
\end{Def}

We can then prove the following generalization of Theorem~\ref{thm:rel-Groth}, in which we do not assume that $f^*$ satisfies Grothendieck-Neeman duality, but only that its source and target do, with respect to their base.

\begin{Thm}
\label{Thm:rel-rel-Groth}%
Let $f^*: \cat D\to \cat C$ be a morphism of $\cat B$-categories (Def.\,\ref{def:B-mor}). Assume that the structure morphisms $p^*:\cat B\to \cat C$ and~$q^*:\cat B\to \cat D$ satisfy Grothendieck-Neeman duality (Thm.\,\ref{thm:GN-intro}). Let $\cat B_0\subset \cat B$ be a $\cat B^c$-subcategory with dualizing object~$\kappa\in\cat B_0$ (Def.\,\ref{def:dualizing_object}). Let $\cat C_0=p^\#\cat B_0$ and $\cat D_0=q^\#\cat B_0$ be its compact pullbacks in~$\cat C$ and $\cat D$ respectively (Def.\,\ref{def:cpt_pullback}), which admit the dualizing objects
$$
\gamma:=\omega_p\otimes p^*(\kappa)\in\cat C_0
\qquadtext{and}
\delta:=\omega_q\otimes q^*(\kappa)\in\cat D_0
$$
respectively, by Theorem~\ref{thm:rel-Groth}. Then $f^\#(\cat D_0)=\cat C_0$ and $f_*$ restricts to a well-defined exact functor $f_*:\cat C_0\to \cat D_0$ which is duality-preserving with respect to $\dual_{\gamma}$ and $\dual_{\delta}$.
\end{Thm}

\begin{proof}
Note that $\gamma=\omega_p\otimes p^*(\kappa)\cong p\uu{1}(\kappa)\cong f\uu{1}q\uu{1}(\kappa)\cong f\uu{1}(\delta)$. So we already know from Theorem~\ref{thm:duality_preserving} that $f_*$ is compatible with the dualities~$\dual_\gamma$ and~$\dual_\delta$. By Proposition~\ref{prop:trans_pullback}, we know that $\cat C_0=p^\#\cat B_0=f^\#q^\#\cat B_0=f^\#\cat D_0$. It follows from this and the definition of $f^\#\cat D_0$ that $f_*(\cat C_0)\subseteq\cat D_0$ yielding the desired $f_*:\cat C_0\to \cat D_0$.
\end{proof}

An example of the above relative discussion over a base category~$\cat B$ is the situation where $\cat B_0=\cat B^c$ and $\kappa=\unit$. In other words, we can assume that the base is sufficiently simple that the duality question over~$\cat B$ is solved in the ``trivial" way, as in Example~\ref{Exa:trivial_duality}. This is interesting in algebraic geometry when $\cat B=\Dqc(S)$ for $S$ regular, as we saw in Corollary~\ref{cor:cpt-pb-AG}, for instance when $S=\Spec(\kk)$ for $\kk$ a field. In that case, it is not a restriction to consider the trivial duality on~$\cat B$, with $\cat B_0=\cat B^c$ and~$\kappa=\unit$. We can then pull it back to obtain a more interesting subcategory with duality~$\cat C_0=p^\#(\cat B_0)$ in~$\cat C$.

\begin{Def}
\label{def:rel-comp}%
Let $\cat C$ be a \emph{$\cat B$-category} as in Definition~\ref{def:B-cat}. We define the full subcategory of~$\cat C$ of \emph{$\cat B$-relatively compact} objects to be
\[
\cat C^{c/p}:= p^\#(\cat B^c) = \SET{ x\in \cat C}{p_*(c\otimes x) \in \cat B^c \textrm{ for all compact } c \in \cat C^c} \,.
\]
\end{Def}

\begin{Exa}
Let $p:X\to S$ be a projective morphism with $S$ regular and let $p^*:\cat B=\Dqc(S)\to \Dqc(X)=\cat C$. Then $\cat C^{c/p}=\Db(\coh X)$ by Corollary~\ref{cor:cpt-pb-AG}.
\end{Exa}

\begin{Cor}
\label{cor:B-Groth}%
Let $f^*: \cat D\to \cat C$ be a morphism of $\cat B$-categories (Def.\,\ref{def:B-mor}) with structure morphisms $p^*:\cat B\to \cat C$ and~$q^*:\cat B\to \cat D$.
\begin{enumerate}[\indent\rm(a)]
\item
Compact pullback (Def.\,\ref{def:cpt_pullback}) preserves the subcategories of $\cat B$-relatively compact objects: $f^\#(\cat D^{c/q})=\cat C^{c/p}$.
\item
Suppose that $p^*:\cat B\to \cat C$ satisfies Grothendieck-Neeman duality. Then $\omega_p=p\uu{1}(\unit)$ is a dualizing object for the subcategory of relatively compact objects~$\cat C^{c/p}$ (Def.\,\ref{def:rel-comp}).
\item
\label{it:B-Groth}%
Suppose that $p^*$ and~$q^*$ satisfy Grothendieck-Neeman duality. Then the functor $f_*$ restricts to an exact functor $f_*:\cat C^{c/p}\to \cat D^{c/q}$ which is duality-preserving with respect to the dualities $\dual_{\omega_p}$ and~$\dual_{\omega_q}$.
\end{enumerate}
\end{Cor}

\begin{proof}
Proposition~\ref{prop:trans_pullback} gives~(a). Theorem~\ref{thm:rel-Groth} applied to $p^*$ gives~(b). Theorem~\ref{Thm:rel-rel-Groth} gives~(c).
\end{proof}

\begin{center}
*\ *\ *
\end{center}

Historically, Grothendieck duality arose in order to generalize Serre duality to the relative situation (\ie to morphisms of schemes).  However, another generalization of Serre duality is given by the notion of a \emph{Serre functor} \cite{boka89, BondalOrlov01}, and we next explain how a relative version of Serre functors naturally arises in our theory.

\begin{Rem} \label{Rem:adj_enrichment}
If $f^*: \cat D\adjto \cat C : f_*$ is an adjunction between closed tensor categories with $f^*$ a tensor functor, then $\cat C$ inherits an enrichment over~$\cat D$ (see~\cite{Kelly05}): the Hom-objects are given by $\Crich (x,y) := f_*\,\ihomcat{C}(x,y)\in \cat D$, and the unit and composition morphisms  $\unit_\cat D \to \Crich(x,y)$ and $\Crich(y,z)\otimes_\cat D \Crich(x,y) \to \Crich(x,z)$ in $\cat D$ are obtained by adjunction in the evident way from the $\cat C$-internal unit and composition maps.
\end{Rem}

\begin{Thm}[Relative Serre duality] \label{Thm:Serre_duality}
Let $f^*: \cat D\to \cat C$ be a functor as in Hypothesis~\ref{hyp:base} and let
$\Crich$ denote
the resulting $\cat D$-enriched category as in Remark~\ref{Rem:adj_enrichment}.
Then there is a canonical natural isomorphism in~$\cat D$
\begin{align} \label{eq:Serre_iso}
\sigma_{x,y} : \dual\Crich(x,y) \stackrel{\sim}{\to} \Crich (y , x \otimes \DO)
\end{align}
for all $x\in \cat C^c$ and $y\in \cat C$,
where we recall $\dual:=\ihomcat{D}(-,\unit)$.
In particular, if we have the Wirthm\"uller isomorphism (Theorem~\ref{thm:Wirthmueller}), the pair $(\mathbb S:= (-)\otimes \DO, \sigma)$ defines a \emph{Serre functor on $\cat C^c$ relative to $\cat D^c$}, by which we mean that $\mathbb S$ is an equivalence $\mathbb S: \cat C^c \stackrel{\sim}{\to} \cat C^c$ and that $\sigma$ is a natural isomorphism $\dual\Crich(x,y)\cong \Crich(y, \mathbb S x)$  in the tensor-category $\cat D^c$ for all $x,y\in \cat C^c$.
\end{Thm}

\begin{proof}
Under our basic hypothesis, we have the adjunction $f_*\dashv f\uu{1}$ and its internal version. If $x$ is compact, and hence rigid, we obtain an isomorphism
\begin{align*}
\dual \,\Crich(x,y)
&\;\;=\;\; \ihomcat{D}\big( f_*\,\ihomcat{C}(x,y) , \unit \big) && \textrm{by definition}\\
&\;\;\cong\;\; f_*\,\ihomcat{C}\big( \ihomcat{C}(x,y) , \DO  \big) && \eqref{eq:inthom-f_*-f^1} \\
&\;\;\cong\;\; f_*\,\ihomcat{C}\big(  y , x \otimes \DO  \big) &&  x\in \cat C^c  \\
&\;\; = \;\; \Crich (y, x\otimes \DO) &&
\end{align*}
(the second isomorphism uses~\cite[Thm.\ A.2.5]{HoveyPalmieriStrickland97} again). This is the claimed natural isomorphism~$\sigma$. When the object~$\DO$ is invertible (Thm.\,\ref{thm:Wirthmueller}), the functor $\mathbb S=(-)\otimes \DO$ restricts to a self-equivalence on compacts.
\end{proof}

\begin{Rem}
Usually, what one means by a ``Serre functor'' is a self-equivalence~$\mathbb S$ on a $\kk$-linear (triangulated) category~$\cat C$ together with an isomorphism as in~\eqref{eq:Serre_iso}, where $\kk$ is a field and $\dual$ should be replaced by the $\kk$-linear dual. We can easily deduce such a structure from our result when the target category is $\cat D=\Der(\kk)$.
\end{Rem}

\begin{Cor}[Serre duality] \label{cor:Serre_duality_over_k}
Let $f^*: \cat D\to \cat C$ satisfy the Wirthm\"uller isomorphism (Theorem~\ref{thm:Wirthmueller}), and assume moreover that $\cat D=\Der(\kk)$ is the derived category of a field~$\kk$.
Then $\cat C^c$ is  $\kk$-linear and endowed with a Serre functor
\[
\mathbb S=(-)\otimes \DO : \cat C^c \stackrel{\sim}{\to} \cat C^c
\quad\quad
\sigma: \cat C(x,y)^\star \stackrel{\sim}{\to} \cat C(y, \mathbb S x)
\]
in the sense of~\cite{boka89, BondalOrlov01}, where $(-)^\star=\Hom_\kk(-,\kk)$ denotes the $\kk$-linear dual.
\end{Cor}

\begin{proof}
Apply $H^0= \Der(\kk)(\unit , -)$ to~\eqref{eq:Serre_iso} and note that $H^0 \circ \dual \cong (-)^\star\circ H^0$.
\end{proof}

\begin{Rem} \label{Rem:Serre_uniqueness}
If it exists, a Serre functor $(\mathbb S, \sigma)$ on $\cat C^c$ relative to $\cat D^c$ as in Theorem~\ref{Thm:Serre_duality} is unique. More precisely, if $(\mathbb S,\sigma)$ and $(\mathbb S',\sigma')$ are two of them then by Yoneda there is a unique isomorphism $\smash{\mathbb S\stackrel{\sim}{\to} \mathbb S' }$ of functors $\cat C^c\to \cat C^c$ inducing $(\sigma'\sigma^{-1})_*$ on the Hom sets.
A similar remark holds for the usual Serre functors.
\end{Rem}

\begin{Exa}[Projective varieties] \label{Exa:proj_var}
Let $p: X\to \Spec(\kk)$ be a \emph{projective} variety over a field~$\kk$, and let $p^*:\cat B=\Der(\kk)\to \cat C=\Der(\Qcoh X)$ be the pull-back functor.  Then $\cat C^c=\Dperf(X)$, and moreover $\cat C^{c/p}=\Db(\coh X)$ by Theorem~\ref{thm:cpt-pb-AG}.  Thus we have the inclusion $\cat C^c = \Dperf(X)\subset \Db(\coh X)= \cat C^{c/p} \smash{\stackrel{\textrm{def.}}{=}} p^\#(\cat B^c)$ and therefore $p^*$ must satisfy Grothendieck-Neeman duality by Proposition~\ref{Prop:Crelcpt_thick}.  We conclude that $p$ is quasi-perfect.  Moreover, by Theorem~\ref{Thm:duality_iso} we know that the subcategory $\Db(\coh X)$ consists of $\omega_p$-reflexive objects in~$\Der(\Qcoh X)$.  Hence by our Grothendieck duality Theorem~\ref{thm:rel-Groth}, the object $\omega_p$ is dualizing for~$\Db(\coh X)$, \ie in more classical language, it is a \emph{dualizing complex for~$X$} (as defined in \cite{Neeman10}). Can we describe~$\omega_p$ more explicitly?

If $X$ is \emph{Gorenstein} (e.g.\ regular, or a complete intersection), then by \cite[p.\,299]{Hartshorne66} the structure sheaf $\mathcal O_X$ is also a dualizing complex for~$X$.  But then, by the uniqueness of dualizing complexes (see Corollary~\ref{cor:unique-AG}), there exists a tensor invertible $\ell\in \Der(\Qcoh X)$ and an isomorphism $\omega_p \cong \mathcal O_X\otimes \ell = \ell$, so in this case $\omega_p$ is invertible and therefore $p^*$ satisfies the Wirthm\"uller isomorphism (Thm.\,\ref{thm:Wirthmueller}). Indeed, it can be shown in general that Gorenstein varieties are characterized by having an invertible dualizing complex (see \cite[$\S$8.3]{AvramovIyengarLipman10}).  Still, this does not yet determine $\omega_p$ up to isomorphism.

Assume further that $X$ is \emph{regular}, so that we have the equality $\cat C^c=\Dperf(X) = \Db(\coh X) = \cat C^{c/p}$.  In this case,  by condition~\eqref{it:rel-cpt} of Theorem~\ref{thm:Wirthmueller}, $\omega_p$ must be invertible.  Moreover, Theorem~\ref{cor:Serre_duality_over_k} applies so that $-\otimes \omega_p$ yields a Serre functor on~$\cat C^c$.  But it is a basic classical result that $\cat C^c$ also admits a Serre functor $-\otimes \Sigma^n\omega_X$, where $\omega_X= \Lambda^n \Omega_{X/\kk}$ is the canonical sheaf on~$X$  (see e.g.\ \cite[Lemma~4.18]{Rouquier10}); here we assume  $X$ is of pure dimension~$n$, for simplicity.  Therefore $\omega_p\cong \Sigma^n\omega_X$ by Remark~\ref{Rem:Serre_uniqueness}.

Suppose now that $f\colon X\to Y$ is a $\kk$-morphism of projective varieties. By Corollary~\ref{cor:B-Groth}\,\eqref{it:B-Groth}, we have a well-defined $f_*:\Db(\coh X)\to \Db(\coh Y)$ compatible with the dualities $\dual_{\omega_p}$ and $\dual_{\omega_q}$ as discussed above (for $q:Y\to \Spec (\kk)$).
\end{Exa}

This example illustrates how our abstract notions and results specialize to ones that are familiar to algebraic geometers, depending on the various additional assumptions that are available.

\bigbreak
\section{Examples beyond Grothendieck duality}
\label{se:Beyond}%

In this final section we show that Matlis duality as well as Neeman's version of Brown-Comenetz duality also fall under the scope of our theory.

\begin{Thm} \label{Thm:abstract_DGI_duality}
Let $f^*: \cat D\to \cat C$ be a functor satisfying our basic Hypothesis~\ref{hyp:base}.
Let $\cat C_0$ be a subcategory of $\cat C$ admitting a dualizing object~$\kappa' \in \cat C$ (external or not, see Def.\,\ref{def:dualizing_object}).
Assume moreover that $\kappa'$ admits a \emph{Matlis lift}~$\kappa$, that is, an object $\kappa \in \cat D$ such that $f\uu{1}(\kappa)\cong \kappa'$.
Then $\kappa$ is a possibly external dualizing object for the subcategory $\cat D_0:= \mathrm{thick}(f_*(\cat C_0))$, the thick subcategory generated by the image of $\cat C_0$ under push-forward.
\end{Thm}

\begin{proof}
Since $\Idcat{D}$ and $\dual^2_{\kappa'}$ are triangulated functors, it suffices to show that the natural transformation $\varpi_{f_*(x)}\colon f_*(x)\to \dual_{\kappa'}^2(f_*x)$ of~\eqref{eq:can_kappa} is invertible whenever $x$ belongs to~$\cat C_0$. For this, recall that $\varpi_{f_*(x)}$ appears in the commutative square~\eqref{eq:duality_preserving}. Since the top horizontal map in~\eqref{eq:duality_preserving} is invertible for $x \in \cat C_0$, the commutativity of~\eqref{eq:duality_preserving} implies that $\varpi_{f_*(x)}$ is also invertible, as desired.
\end{proof}

\begin{Exa}[Matlis duality] \label{Exa:matlis}
Let $R$ be a commutative noetherian local ring, let $R\to k$ be the quotient map to the residue field and let $f^*:\Der(R)\to \Der(k)$ be the induced functor as in Example~\ref{Exa:AG_affine}. Then $E(k)$, the injective hull of the $R$-module~$k$, is a Matlis lift of~$k$:
$f\uu{1}(E(k))=\mathrm{RHom}_R(k,E(k))\cong k$ in $\Der(k)$.
By Theorem~\ref{Thm:abstract_DGI_duality}, the functor $\dual_{E(k)}=\mathrm{RHom}_R(-,E(k))$ induces a duality on the thick subcategory of $\Der(R)$ generated by $f_*(k)$, \ie on complexes whose homology is bounded and consists of finite length modules.
As $E(k)$ is injective, we may restrict this duality to the category of finite length modules.
\end{Exa}
\begin{Exa}[Pontryagin duality] \label{Exa:pontryagin}
The dualizing object $E(k)$ of Example~\ref{Exa:matlis} is typically external, \ie it often lies outside the subcategory it dualizes: $E(k)\not\in \mathrm{thick}(f_*(k))$. This already happens in the archetypical example of (discrete $p$-local) Pontryagin duality, where $R\to k$ is the quotient map $\bbZ_{(p)}\to \bbZ/p$ and $E(k)$ is the Pr\"ufer group $\bbZ[\frac{1}{p}]/\bbZ\cong\bbQ/\bbZ_{(p)}$,
which has infinite length.
\end{Exa}

\begin{Exa}[Generalized Matlis duality] \label{Exa:DGI}
Let $R\to k$ be a morphism of commutative $\Sphere$-algebras and consider the three induced functors
\begin{equation}
\vcenter{\xymatrix{
\quad\quad\quad\quad \cat C:=\Der(k)
  \ar@{<-}[d]_-{\!\Displ k \otimes_R(-) \,=\, f^{*}}
  \ar@<17pt>[d]^-{\Displ f_*}
  \ar@<37pt>@{<-}[d]^{\Displ f\uu{1}\,=\, \Hom_R(k,- )  }
  \ar@<67pt>@{}[d]
    \\
\quad\quad\quad\quad \cat D:= \Der(R)
}}
\end{equation}
as in Example~\ref{Exa:Brave_New}.
We write $R\to k$ rather then $B\to A$ in order to be consistent with the notation of Dwyer-Greenlees-Iyengar~\cite{DwyerGreenleesIyengar06}.
In \emph{loc.\,cit.,} a \emph{Matlis lift of $k$} is defined to be a (structured) $R$-module~$I$ such that $\Der(k)(x,k) \cong \Der(R)(f_*x, I)$ naturally in $x\in \Der(k)$, \ie by Yoneda, such that $f\uu{1}(I)\cong k$ (see \cite[Def.\,6.2 and Rem.\,6.3]{DwyerGreenleesIyengar06}).
Moreover, $I$ is required to be ``effectively constructible from~$k$'', a property somewhat stronger than $I$ belonging to the localizing subcategory generated by $f_*(k)$.
In particular, a Matlis lift of~$k$ in the sense of Dwyer-Greenlees-Iyengar is also a Matlis lift, in the more modest sense of Theorem~\ref{Thm:abstract_DGI_duality}, of the dualizing object $\kappa':=k$ for the subcategory $\cat C_0:=\Der(k)^c$ of~$\Der(k)$. Hence by Theorem~\ref{Thm:abstract_DGI_duality} we immediately obtain the following generalization of Matlis duality.
\end{Exa}

\begin{Cor} \label{Cor:DGI}
Let $R\to k$ be a morphism of commutative $\Sphere$-algebras, and assume that the $R$-module $I$ is a Matlis lift of~$k$ in the sense of \cite{DwyerGreenleesIyengar06}.
Then $I$ is a (possibly external) dualizing object for the thick subcategory of $\Der(R)$ generated by~$f_*(k)$.
\qed
\end{Cor}

Similar ideas and results can also be found in \cite{DwyerGreenleesIyengar11} \cite{Yekutieli10}~\cite{PortaShaulYekutieli14}.

\begin{center}
*\ *\ *
\end{center}

\begin{Exa}[Brown-Comenetz duality]
Let $\SH$ denote the stable homotopy category of spectra and recall that the Brown-Comenetz dual
$I_{\bbQ/\bbZ}E$ of a spectrum $E \in \SH$ is defined, using Brown Representability, by the equation
\begin{equation}
	\label{eq:BC-defn}
\Hom_{\bbZ}(\pi_0(X\wedge E), \bbQ/\bbZ) \cong \Hom_{\SH}( X, I_{\bbQ/\bbZ}E) \qquad (X\in \SH).
\end{equation}
It follows immediately from the tensor-Hom adjunction that $I_{\bbQ/\bbZ}E$ is given by the function spectrum $\dual_{I_{\bbQ/\bbZ}}(E) = \ihom_{\SH}(E,I_{\bbQ/\bbZ})$ where $I_{\bbQ/\bbZ}:=I_{\bbQ/\bbZ}\Sphere$ denotes the Brown-Comenetz dual of the sphere.
Classical Brown-Comenetz duality asserts that the functor
$\dual_{I_{\bbQ/\bbZ}} = \ihom_{\SH}(-,I_{\bbQ/\bbZ}):\SH\op\to \SH$
restricts to a duality on the subcategory of ``homotopy finite'' spectra, \ie spectra whose homotopy groups are all finite.
Moreover, defining $I_{\bbZ}$ to be the homotopy fiber of the canonical map $H\bbQ \to I_{\bbQ/\bbZ}$, we obtain a functor $\dual_{I_{\bbZ}} = \ihom_{\SH}(-,I_{\bbZ})$ which restricts to a duality on the larger subcategory consisting of those spectra whose homotopy groups are all finitely generated (cf.~\cite[\S 2]{HeardStojanoska14}).  Restricted to the subcategory of homotopy finite spectra, this ``Anderson duality'' agrees with Brown-Comenetz duality up to a shift: $\dual_{I_{\bbQ/\bbZ}}\cong\Sigma\dual_{I_{\bbZ}}$. 

On the other hand, the main result of \cite{Neeman92} establishes the existence of a (unique) triangulated functor $\Pi : \SH \to \Der(\bbZ[\frac{1}{2}])$ which lifts stable homotopy with~2 inverted: $H^0 \circ \Pi = \pi_0(-)[\frac{1}{2}]$.  
Neeman defines a new ``Brown-Comenetz type'' dual $\widetilde{I}E$ by the equation 
\begin{equation} \label{eq:defNBC}
\Hom_{\Der(\bbZ[\frac{1}{2}])}(\Pi(X\wedge E), \bbZ[\textstyle\frac{1}{2}]) \cong \Hom_{\SH}( X, \widetilde{I}E) \qquad (X\in \SH).
\end{equation}
As before, we immediately see that $\widetilde{I}E = \dual_{\widetilde{I}\Sphere}(E)$.  The next theorem demonstrates that Neeman's $\dual_{\widetilde{I}\Sphere}$ delivers a version of Brown-Comenetz duality with the same scope as Anderson duality -- and with a completely conceptual origin -- provided we also invert the prime 3. To this end, regard $\Pi$ as a functor $\SH[\frac{1}{6}]\to \Der(\bbZ[\frac{1}{6}])$ where $\SH[\frac{1}{6}]$ denotes the stable homotopy category localized away from $6$; being a finite localization of the ordinary stable homotopy category~$\SH$, it is again a rigidly-compactly generated tensor-triangulated category generated by its unit.

Neeman proves that, even before 3 is inverted, the Moore spectrum construction extends to an exact functor $f^*$ (denoted by $F$ in \emph{loc.\,cit.}) which is left adjoint to $\Pi$ and admits a natural isomorphism $\mu: f^*(x) \wedge f^*(y)\cong f^*(x\otimes y)$ (\cite[Prop.\,3.6 and 5.5]{Neeman92}).  He then shows that $\mu$ satisfies the hexagon axiom of a monoidal functor if and only if 3 is also inverted (\cite[Prop.\,5.6 and Ex.\,5.1]{Neeman92}), and derives from this the second main result of his article: $\SH[\frac{1}{6}]$ admits an enrichment over $\Der(\bbZ[\frac{1}{6}])$.  Although he never states so, he really shows that the functor $f^*:\Der(\bbZ[\frac{1}{6}])\to \SH[\frac{1}{6}]$ is symmetric monoidal, with structure map $f^*(\unit)\to \unit$ given by the identity map (in fact he deduces the enrichment from the monoidal adjunction $f^*\dashv f_*$ as in Remark~\ref{Rem:adj_enrichment}).  Since $f^*$ has a right adjoint it preserves coproducts, so it satisfies Hypothesis~\ref{hyp:base} and we deduce the existence of $f\uu{1}$ by Corollary~\ref{cor:base}, as usual:
\begin{equation*}
\vcenter{\xymatrix@R=3em{
\SH[\frac{1}{6}]
  \ar[d]|-{\Displ f_* = \Pi }
    \\
\Der(\bbZ[\frac{1}{6}])
  \ar@<-23pt>[u]_-{\Displ f\uu{1}}
  \ar@<23pt>[u]^-{\Displ f^{*}}
}}
\end{equation*}
\end{Exa}

\begin{Thm} \label{thm:NBCduality}
The above triple $f^*\dashv f_* \dashv f\uu{1}$ has the following properties:
\begin{enumerate}[\indent\rm(a)]
\item \label{it:DO_NBC}
The relative dualizing object  associated with~$f^*$ is Neeman's Brown-Comenetz dual of the sphere: $\DO = \widetilde{I}\Sphere $.
Hence we have $\widetilde{I}E = \dual_{\DO}(E)$ for all~$E\in \SH[\frac{1}{6}]$.
\item \label{it:scope_of_NBC}
The object $\widetilde{I}\Sphere$ is dualizing for the subcategory $\cat C_0\subset \SH[\frac{1}{6}]$ of spectra whose homotopy groups are finitely generated $\bbZ[\frac{1}{6}]$-modules. In other words, Neeman's Brown-Comenetz duality restricts to an equivalence
$\widetilde{I}(-): (\cat C_0)\op\stackrel{\sim}{\to}\cat C_0$.
\item \label{it:Pi_preserves_dualities}
There is a canonical natural isomorphism
\[\Pi(\widetilde{I}E)
= f_* \dual_{\DO}(E)
\cong
\dual_{\unit} (f_* E) =
\ihom_{\Der(\bbZ[\frac{1}{6}])}(\Pi E,\bbZ[\textstyle \frac{1}{6}])\]
for all $E\in \SH[\frac{1}{6}]$, analogous to the isomorphism
\[\pi_*(I_{\bbQ/\bbZ}E)\cong \Hom_{\bbZ}(\pi_*E,\bbQ/\bbZ)\]
of ordinary Brown-Comenetz duality.
\item \label{it:GNvsNBC}
$f^*$ does not satisfy Grothendieck-Neeman duality:
$f\uu{1}\not\cong \widetilde{I}\Sphere \wedge f^*$.
\end{enumerate}
\end{Thm}

\begin{proof}
Part~\eqref{it:DO_NBC} holds simply because $\DO$ and $\widetilde{I}\Sphere$
are defined by the same natural isomorphism:~\eqref{eq:defNBC} with $E=\Sphere$.
Part~\eqref{it:Pi_preserves_dualities} is~\eqref{eq:inthom-f_*-f^1}.
Part~\eqref{it:GNvsNBC} holds because $f_*=\Pi$ does not preserve compact objects.
Indeed, $\pi_*(\Sphere)_{(p)}$ is well-known to be unbounded for every prime~$p$ (see e.g.~\cite{McGibbonNeisendorfer84}).
Finally, to prove~\eqref{it:scope_of_NBC}, consider the full subcategory $\cat D_0\subset \cat D=\Der(\bbZ[\frac{1}{6}])$ of complexes whose homology groups are finitely generated $\bbZ[\frac{1}{6}]$-modules. As the ring $\bbZ[\frac{1}{6}]$ is principal (hence hereditary), every object of its derived category has (non-canonically) the form $M=\coprod_i \Sigma^i M_i$ for modules~$M_i$, and every morphism between such objects has nonzero components only of degree $0$ and~$-1$. An object~$M=\coprod_i \Sigma^i M_i$ belongs to the subcategory $\cat D_0$ if and only if every~$M_i$ is finitely generated. Using that the natural map $\coprod_i \Sigma^i M_i\isoto \prod_i \Sigma^i M_i$ is an isomorphism here, it is easy to see that $\ihom(-,\unit)$ is a duality on~$\cat D_0$. By Theorem~\ref{Thm:duality_iso}\,\eqref{it:i=>ii} for $\kappa=\unit$, using that our $\cat C_0$ is nothing but~$f^\#(\cat D_0)$, we get the result.
\end{proof}

\begin{Rem} \label{Rem:no_f}
As explained in \cite[Rem.\,4.2]{Neeman92}, there are K-theoretic obstructions for the functor $f^*:\Der(\bbZ[\frac{1}{2}]) \to \SH[\frac{1}{2}]$
 to derive from a functor of the underlying (Waldhausen) model categories.
This shows that there can be interesting functors $f^*$ with no underlying map~$f$.
\end{Rem}

\subsection*{Acknowledgements:}
We thank Greg Stevenson for helpful discussions, Henning Krause for the reference to~\cite{Jantzen87}, and an anonymous referee for a careful reading and several useful comments.

\bibliographystyle{alpha}%
\bibliography{TG-articles}

\end{document}